\documentclass[12pt]{amsart}
\usepackage{amsmath,amscd,amsthm,amssymb,enumerate, url, eucal,  rotating}
\usepackage{fullpage}

\newcommand{\negpar}[1][-1em]{%
  \ifvmode\else\par\fi
  {\parindent=#1\leavevmode}\ignorespaces
}

\swapnumbers
\theoremstyle{plain}
\newtheorem{theorem}{Theorem}[section]
\newtheorem{lemma}[theorem]{Lemma}

\newtheorem{thm:1936a}[theorem]{Theorem\,\normalfont(Whitehead(1936a))}
\newtheorem{cor:nielsen}[theorem]{Corollary\,\normalfont(Nielsen(1919))}

\theoremstyle{definition}

\newtheorem{definitions}[theorem]{Definitions}

\newtheorem{notation}[theorem]{Notation}

\numberwithin{equation}{section}

\newcommand{\leftwreath}{\text{{\rotatebox{90}{$\scriptstyle\curvearrowright$}}}\mkern2mu}
\newcommand{\rightwreath}{\mkern2mu\text{{\reflectbox{\rotatebox{90}{$\scriptstyle\curvearrowright$}}}}}
\newcommand{\rank}{\operatorname{rank}}

\newcommand{\EE}{\operatorname{E}}
\newcommand{\VV}{\operatorname{V}}

\newcommand{\ZZ}{{\mathbb Z}}

\def \naturals{\mathbb{N}}
\def \integers {\mathbb{Z}}

\def\d1{\discretionary{-}{}{-}}
\def\e1{\discretionary{--}{}{--}}

\renewcommand{\phi}{\varphi}
\renewcommand{\le}{\leqslant}
\renewcommand{\ge}{\geqslant}

\begin{document}

\title{A graph-theoretic proof for\\
Whitehead's   second free\d1group algorithm}

\author{Warren Dicks*}
\address{Departament de Matem{\`a}tiques,  \newline  Universitat Aut{\`o}noma de Barcelona,
\newline E-08193 Bellaterra (Barcelona),  \newline  SPAIN \newline\null}
\email{dicks@mat.uab.cat \newline  \normalfont{\emph {Home page}:\,\,\,\,}\url{http://mat.uab.cat/~dicks/}}
\thanks{*Research supported by
MINECO (Spain) through project number MTM2014-53644}
\date{}


\begin{abstract}   J.\,H.\,C.\,Whitehead's second free-group algorithm
 determines whether or not two given elements  of
a  free group lie in the same orbit of the automorphism group of
the free group.  The algorithm involves certain connected graphs, and
Whitehead used  three-manifold models to prove their
 connectedness; later,
Rapaport  and Higgins\hskip2pt{\scriptsize\&}\hskip2ptLyndon
 gave    group-theoretic proofs.

 Combined work of Gersten, Stallings, and Hoare  showed that
the three-manifold  models  may be viewed as  graphs.  We   give the direct translation of
 Whitehead's topological argument into the language of graph theory.
\end{abstract}

\maketitle

\section{Minimal background}\label{sec:outline}

Whitehead(1936b) gave an algorithm which, with input two finite sequences \mbox{$S_1,\,S_2$} of
 elements (or  conjugacy classes of elements) of a finite-rank  free group~$F$,
outputs either an $F$-automorphism $\phi$
such that \mbox{$\phi(S_1) = S_2$} or an assurance that no such~$\phi$ exists.
More importantly, he  introduced certain connected graphs that have
  been of great interest to group theorists.  His
nine-page proof of  connectedness  used a three-manifold model for each $F$-automorphism.
Rapaport(1958)  gave a twenty-page group-theoretic  proof of connectedness, and
Higgins\hskip2pt{\scriptsize\&}\hskip2ptLyndon(1962,\,1974)
gave one of  five pages; these proofs led the way to an even deeper understanding of
$F$-automorphisms.

Gersten(1987)  constructed  a    graph model for each $F$-automorphism, and
 Stallings(1983) pointed out a connection between Gersten's model   and Whitehead's.
 Krsti\'c(1989) used Cayley trees to simplify Gersten's construction.
Hoare(1990) gave an explicit description of Whitehead's model
in terms of Gersten's.  Below, we  give the resulting translation of
 Whitehead's topological argument into the language of
 graph theory\mbox{.\hskip-3pt}\footnote{\hskip2ptFor his earlier free-group algorithm, Whitehead(1936a)
also  used  three-manifold models, to prove his celebrated cut\-ver\-tex lemma.
Hoare(1988) gave the second proof,
 using Gersten's graphs
 in place of    manifolds.
 Dicks(2014, 2017),
refining work of Stong(1997), proved a more general  result by   tricolouring a Cayley tree.
The elegant folding theorem of Heusener\hskip2pt{\scriptsize\&}\hskip2ptWeidmann(2014)
 leads to a yet more general result.}\hskip5ptThis argument  concerns changes of
 bases  (free-generating sets) rather than automorphisms, and ours may be the
 first treatment of Gersten's graphs  that does not
  mention   group morphisms.

All of the following will apply throughout.

\begin{notation}\label{not:movesbis} Set \mbox{$\naturals:= \{0,1,2,\ldots \}$}.

Let $F$ be a  finite-rank  free group.
By a \textit{straight word in $F$}, we mean an element of $F$;
by a \textit{cyclic word in $F$}, we mean the $F$-conjugacy class of an element  of $F$; and,
by a \textit{word in~$F$}, we mean a straight-or-cyclic word in $F$.
Let $R$ be   a   finite set  of words in~$F$. Let $X$~and~$Y$ be $F$-bases.
In Section~\ref{sec:freegroups}, we shall recall the value
\mbox{$h(X):= \sum_{r \in R} X\text{-length}(r)\in \naturals$}.
We write \mbox{$X^{\pm 1}:= X \cup X^{-1}$}.
 We say that
\mbox{$Y$} is a  \textit{Whitehead transform of~$X$}
if there exists
 some \mbox{$x \in   X^{\pm1}$} such that
\mbox{$Y  \subseteq  \{1,   x  \}{\cdot}X   {\cdot} \{1,x^{-1}\}.$}
We say that   \textit{$X$~is a local-minimum point for~$h$}  if
\mbox{$h(X) \le h(X')$}
for each Whitehead transform $X'$ of~$X$.
  \qed
\end{notation}

In Section~\ref{sec:gersten}, we shall use Gersten's graphs to define a
 value \mbox{$\operatorname{d}(X,Y) \in \naturals$}   that
 Whitehead used tacitly. What  the  topological portion of Whitehead's argument   shows is precisely
 \begin{align}
&\text{if $X$~and~$Y$ are local\d1min\-imum points for $h$,   then either
  \mbox{$X^{\pm 1}  =   Y^{\pm 1}$} or some}\label{eq:3.3}
\\[-1mm]&\text{Whitehead transform  $Y'$ of $Y$
 satisfies \mbox{$h(Y') = h(Y)$} and \mbox{$\operatorname{d}(X,Y') < \operatorname{d}(X,Y)$}.}\phantom{---}\nonumber
 \end{align}
 This will be  stated  as Theorem~\ref{th:main}  below,  and
 our sole objective is to give a  self\d1contained graph-theoretic  proof that copies Whitehead's.
All the other parts of his article are graph theoretic  or group theoretic,
and we shall not discuss them.
However,  Whitehead leaves the main consequence of~\eqref{eq:3.3} unsaid,  and  it is as follows.

Let us say that $Y$ is an \textit{$F\mkern-2mu$-neighbour} of~$X$ if either
 \mbox{$Y^{\pm 1}  =   X^{\pm 1}$} (whence  \mbox{$h(Y) = h(X)$}) or $Y$ is a
Whitehead transform of~$X$.  Let \mbox{$\Gamma(F)$}
denote the graph with vertices the $F$-bases  and with edges joining $F\mkern-2mu$-neighbours.
Let \mbox{$\Gamma(h)$}  denote
 the  subgraph of $\Gamma(F)$  with vertices the local-minimum points for $h$  and
with edges joining $F\mkern-2mu$-neighbours.  It is obvious, but important, that
$h$ is constant on  each connected subgraph of~\mbox{$\Gamma(h)$},
and that    a simple algorithm  outputs a strictly   $h$-decreasing
  \mbox{$\Gamma(F)$}-path   starting at any given \mbox{$\Gamma(F)$}-vertex and  stopping when
 \mbox{$\Gamma(h)$} is reached.
Now suppose that $X$ is a local-minimum point for~$h$  and   that
\mbox{$h(Y) \le h(Z)$} for each \mbox{$F$}-basis~$Z$.
By induction on \mbox{$\operatorname{d}(X,Y)$}, it follows from~\eqref{eq:3.3}  that
 there   exists some ($h$-con\-stant)
\mbox{$\Gamma(h)$}-path  from $Y$~to~\mbox{$X$}.
On varying $X$, we find that
 \mbox{$\Gamma(h)$} is
 connected, which may be considered  to be the main result of Whitehead(1936b);
it greatly generalizes  the result of Nielsen(1919)
that  \mbox{$\Gamma(F)$} itself is
 connected.

The connectedness of \mbox{$\Gamma(F)$}
  was used in the arguments of Whitehead, Rapaport,  Higgins\hskip2pt{\scriptsize\&}\hskip2ptLyndon,
and Gersten.  However,  Krsti\'c did not use it, and this will permit  us to  prove~\eqref{eq:3.3}
without using it.

\section{Review of Cayley  trees}\label{sec:freegroups}\label{sec:graphs}

\begin{definitions}
By a  \textit{graph}, we mean a quintuple $(\, \Gamma, \VV\mkern-2mu\Gamma, \EE\mkern-2mu\Gamma, \iota, \tau)$
such that $\Gamma$ is a set,  \mbox{$\VV\mkern-2mu\Gamma$}~and  \mbox{$\EE\mkern-2mu\Gamma$}~are disjoint
subsets of~$\Gamma$ whose union is \mbox{$\Gamma$}, and
$\iota$ and $\tau$ are  maps from
\mbox{$\EE\mkern-2mu\Gamma$} to~\mbox{$\VV\mkern-2mu\Gamma$}.
We use the same symbol $\Gamma$ to denote both the  graph and the set.
We call \mbox{$\VV\mkern-2mu\Gamma$} and  \mbox{$\EE\mkern-2mu\Gamma$}
 the \textit{vertex-set} and \textit{edge-set} of $\Gamma$
respectively, and  call their elements \textit{$\Gamma$-vertices} and  \textit{$\Gamma$-edges}
respectively.  The maps $\iota$ and $\tau$ are called the \textit{initial}  and
\textit{terminal}  incidence functions  respectively.

Each \mbox{$ e \in  \EE\mkern-2mu  \Gamma $} has an inverse
in the free group  \mbox{$\langle\,\EE\mkern-2mu\Gamma\mid\emptyset\rangle$}, and we set
\mbox{$\iota(e^{-1}):= \tau(e)$} and  \mbox{$\tau(e^{-1}):= \iota(e)$}.
For each \mbox{$v \in \VV\mkern-2mu\Gamma$}, by
 the \textit{$\Gamma$-valence} of $v$,  we mean
 \mbox{$\bigl\vert \{e \in (\EE\mkern-2mu  \Gamma)^{\pm 1} : \iota e= v\}\bigr\vert$}.

By a \textit{\mbox{$\Gamma$}-path},
we mean a  sequence  of the form
 \mbox{$p =  (v_0, e_1, v_1, e_2, v_2, \ldots, v_{\ell-1}, e_\ell, v_\ell)$},
 where
 \mbox{$\ell \in \naturals$} and, for each \mbox{$i \in \{1,2,\ldots, \ell\}$},
\mbox{$e_i \in  (\EE\mkern-2mu  \Gamma)^{\pm 1}$}, \mbox{$v_{i-1} = \iota e_i$},  and \mbox{$v_i = \tau e_i$}.
We sometimes abbreviate $p$ to \mbox{$(e_1,e_2,\ldots, e_\ell)$}, even if \mbox{$\ell = 0$}
when $v_0$ is specified.
The path $p$ is said to be \textit{from $v_0$ to $v_\ell$}, and to have \textit{length}~$\ell$.
 For each \mbox{$ e \in  \EE\mkern-2mu  \Gamma $},  by the \textit{number of times  $p$
traverses~$e$}, we mean \mbox{$\bigl\vert \bigl\{i \in
\{1,2,\ldots, \ell\} : e_i \in \{e\}^{\pm1} \bigr\}\bigr\vert$}.
 We call  the element \mbox{$  e_1 e_2 \cdots e_\ell$} of
 \mbox{$\langle\, \EE \mkern-2mu  \Gamma \mid \emptyset\rangle$}
the \textit{$\Gamma$-label of~$p$}.
If \mbox{$v_\ell = v_0$}, then we say that $p$ is a \textit{closed} path \textit{based at $v_0$}.
If \mbox{$e_{i}\ne e_{i-1}^{-1}$}  for each \mbox{$i \in \{ 2, 3,  \ldots, \ell\}$},
then we say  that \mbox{$p$} is  a \textit{reduced} path.

For   \mbox{$v,w \in \VV\mkern-2mu\Gamma$}, let \mbox{$ \Gamma[v,w]$}
  denote the set of all $\Gamma$-paths from $v$ to $w$;  we then have the \textit{inversion}  map
 \mbox{$ \Gamma[v,w] \to  \Gamma[w,v]$},  \mbox{$p \mapsto p^{-1},$}
where \mbox{$(e_1,e_2, \ldots, e_\ell)^{-1}:= (e_\ell^{-1}, \ldots, e_2^{-1}, e_1^{-1})$}.
For   \mbox{$u,v,w \in \VV\mkern-2mu\Gamma$}, we have the
 \textit{concatenation}  map
 \mbox{$ \Gamma[u,v] \times  \Gamma[v,w]
\to  \Gamma[u,w]$},  \mbox{$(p_1,p_2) \mapsto  p_1\#p_2,$}
where \mbox{$(e_1,e_2, \ldots, e_\ell)\#(e_1',e_2',\ldots, e_m'):=
 (e_1,e_2, \ldots, e_\ell, e_1',e_2',\ldots, e_m').$}
If  a $\Gamma$-path $p$ is closed and \mbox{$p\#p$}  is reduced, we   say  that $p$ is \textit{cyclically reduced}.

We say that $ \Gamma  $ is a
 \textit{tree} if \mbox{$ \VV \mkern-2mu  \Gamma \ne \emptyset$} and, for all
 \mbox{$v,w \in  \VV \mkern-2mu  \Gamma$}, there exists a
unique reduced  $\Gamma$-path  from $v$ to~$w$. We say that $ \Gamma  $ is
\textit{connected} if,
for all \mbox{$v,w \in  \VV \mkern-2mu  \Gamma$}, there exists a
 $\Gamma$-path  from $v$ to~$w$.
By a   \textit{component} of  $\Gamma$,  we mean a maximal nonempty  connected subgraph of~$\Gamma$.  Thus,
 $ \Gamma  $ equals the
disjoint union of its components.
We say that $\Gamma$~is a \textit{forest} if each
component of $\Gamma$ is a tree.  Thus, $\Gamma$ is not a forest if and only if
some closed  $\Gamma$-path traverses some $\Gamma$-edge exactly once.

For any group $G$, we say that $\Gamma$ is   a \textit{left $G$-graph} if
 $\VV\mkern-2mu\Gamma$ and $\EE\mkern-2mu\Gamma$ are left $G$-sets, and $\iota$ and~$\tau$ are left-$G$-set morphisms;
\textit{right $G$-graphs} are defined similarly. \qed
\end{definitions}

Recall that $F$ is a finite-rank  free group, and that $X$ and $Y$ are $F$-bases.  The finite-rank hypothesis
will not be used in this section.

\begin{definitions}
For any \mbox{$g \in F$}, we let \mbox{${\cdot}g$} and \mbox{$g{\cdot}$} denote the permutations
\mbox{$F \to F$} given by \mbox{$v \mapsto vg$} and  \mbox{$v \mapsto gv$}  respectively.
For any subset $S$ of $F$, we write  \mbox{${\cdot}S:= \{{\cdot}g : g \in S\}$} and
\mbox{$S{\cdot}:= \{g{\cdot} : g \in S\} $}.

We let \mbox{$F \rightwreath Y$} denote the (Cayley)
graph with vertex-set \mbox{$F$} and edge-set \mbox{$F \times {\cdot}Y$},
for which each edge \mbox{$(v,{\cdot}y)$} has initial vertex $v$ and\vspace{-1.2mm} terminal vertex $vy$.
The    \mbox{$(F \rightwreath Y)$}-paths \mbox{$(v, (v,{\cdot}y), vy)$}
and \mbox{$(vy, (v,{\cdot}y)^{-1}, v)$} are depicted as
\mbox{$v \xrightarrow{{\cdot}y}{\mkern-9mu-} vy$}  and
\mbox{$vy \xrightarrow{{\cdot}y^{-1}}{\mkern-9mu-} v$} respectively.
An  \mbox{$(F \rightwreath Y)$}-path $p$ will sometimes be depicted in the form\vspace{-1mm}
$$v   \xrightarrow{{\cdot}y_1}{\mkern-9mu-}  vy_1
\xrightarrow{{\cdot}y_2}{\mkern-9mu-} vy_1y_2
\to{\mkern-9mu-} \, \cdots   \to{\mkern-9mu-}
vy_1y_2\cdots y_{\ell-1}\xrightarrow{{\cdot}y_\ell}{\mkern-9mu-}  vy_1y_2 \cdots y_{\ell-1}y_\ell, $$
for a unique \textit{\mbox{$Y^{\pm 1}$}-sequence}   \mbox{$\sigma = (y_1, y_2, \ldots, y_\ell)$},
that is,  an  $\ell$-tuple  of elements of \mbox{$Y^{\pm 1}$}   for some  \mbox{$\ell \in \naturals$}.
We call $\sigma$ the \textit{right \mbox{$Y^{\pm1}$}-label} of $p$.
We say  that $\sigma$  is   \textit{reduced} if  \mbox{$y_i  \ne y_{i-1}^{-1}$}
for each \mbox{$i \in \{2,3,\ldots,\ell\}$},
and  that $\sigma$ is   \textit{cyclically reduced} if
\mbox{$(y_1,y_2,\ldots,y_\ell, y_1,y_2,\ldots,y_\ell)$} is reduced.
Thus, $p$ is a reduced \mbox{$(F \rightwreath Y)$}-path if and only its right
\mbox{$Y^{\pm1}$}-label is a reduced
\mbox{$Y^{\pm 1}$}-sequence.

We let \mbox{$X \leftwreath F$} denote the
graph with vertex-set \mbox{$F$} and edge-set \mbox{$X{\cdot}\times F$},
for which each edge \mbox{$(x{\cdot}, v)$} has initial vertex $v$ and\vspace{-1.2mm} terminal vertex $xv$.
The   \mbox{$(X \leftwreath F)$}-paths \mbox{$(v,(x{\cdot}, v), xv)$}
and   \mbox{$(xv,(x{\cdot},v)^{-1}, v)$} are depicted as
  \mbox{$v \xrightarrow{x{\cdot}}{\mkern-9mu-} xv$}  and
   \mbox{$xv \xrightarrow{x^{-1}{\cdot}}{\mkern-9mu-} v$} respectively.
An \mbox{$(X \leftwreath F)$}-path $p$ will sometimes be depicted in the form\vspace{-1mm}
$$ v    \xrightarrow{x_1{\cdot}}{\mkern-9mu-}  x_1v \xrightarrow{x_2{\cdot}}{\mkern-9mu-} x_2x_1v
\to{\mkern-9mu-} \,\cdots\to{\mkern-9mu-}
 x_{\ell-1} \cdots x_2x_1v \xrightarrow{x_\ell{\cdot}}{\mkern-9mu-}  x_\ell x_{\ell-1} \cdots x_2x_1v, $$
for a unique  \mbox{$X^{\pm 1}$}-sequence    \mbox{$\sigma = (x_\ell, \ldots, x_2, x_1)$}, called
 the \textit{left \mbox{$X^{\pm1}$}-label} of $p$.
Again, $p$ is a reduced \mbox{$(X \leftwreath F)$}-path if and only its left \mbox{$X^{\pm1}$}-label is a reduced
\mbox{$X^{\pm 1}$}-sequence.

We let \mbox{$X \leftwreath F \rightwreath Y$}
denote the graph with vertex-set $F$ and edge-set the (disjoint) union of
 \mbox{$X{\cdot}\times F$} and \mbox{$F \times {\cdot}Y$}, with
initial and terminal vertices as before.  Thus,
\mbox{$X \leftwreath F$} and
\mbox{$F \rightwreath Y$} are subgraphs of
\mbox{$X \leftwreath F \rightwreath Y$} which are   being amalgamated over
their common vertex-set $F$. \qed
\end{definitions}

 Dehn(1910) initiated the study of Cayley graphs of infinite groups, particularly
surface groups, and he  must  have known the following  at the start.

\begin{theorem}\label{th:tree}  The left $F\mkern-2mu$-graph \mbox{$F \rightwreath Y$} is a tree.
\end{theorem}

\begin{proof}[Proof\,\,\normalfont(Fox(1953),  streamlined by Dicks(1980))] Set \mbox{$T:= F \rightwreath Y$}.
For each \mbox{$ (v,y) \in   F \times Y$}, set
 \mbox{$v{\otimes} y := (v,{\cdot}y) \in  F \times  {\cdot} Y = \EE T$};
thus, \mbox{$\iota(v{\otimes} y) = v$} and \mbox{$\tau(v{\otimes} y) = v y$}.

Clearly,  $T$~is  nonempty.

Let $\sim$ denote
the inclusion-smallest equivalence relation   on \mbox{$\VV\mkern-2mu T$} such that
\mbox{$\iota (v{\otimes} y) \sim \tau (v{\otimes} y)$} for each $T$-edge \mbox{$v{\otimes} y$}.
There exists a   left-$F$-set isomorphism between  the set of components of $T$ and the
set of equivalence classes of~$\sim$.
Also, $\sim$\, is
the inclusion-smallest equivalence relation  on \mbox{$F$} such that \mbox{$v \sim v y$}
  for each \mbox{$ (v,y) \in   F \times Y$}.
In particular, the  equivalence class $[1]$ of $1$
satisfies \mbox{$[1]  = [y] = y{\cdot}[1]$}  for each $y \in Y$.  Hence,
the subgroup \mbox{$\{f \in F : f{\cdot}[1] =  [1]\}$}
  of $F$   includes~$Y$.
 Thus, for all \mbox{$f \in \langle Y \,\rangle = F$},
 \mbox{$[1] = f{\cdot}[1] =  [f].$}
Hence,   \mbox{$[1] = F$}.  Thus, $T$~is connected.

For each set $S$, we let \mbox{$\ZZ [S]$} denote   the  free   \mbox{$\ZZ\mkern2mu$}-module
on \mbox{$S$}.
The maps  \mbox{$\iota, \tau: \EE T \to     \VV\mkern-1mu T $}  induce
\mbox{$\ZZ\mkern2mu$}-module morphisms
\mbox{$\hat\iota, \hat\tau: \ZZ [\,\EE T\mkern1mu] \to      \ZZ [\,\VV\mkern-1mu T\mkern1mu]$}.
For each closed $T$-path $p$ which  traverses some $T$-edge exactly once,
the abelianization map \mbox{$\langle\, \EE  T \mid \emptyset \rangle \to  \ZZ [\,\EE T\mkern1mu]$}
carries  the $T$-label of $p$    to a nonzero element of
 the kernel of~\mbox{$\hat\tau{-}\hat\iota$}.  Thus, to show that $T$ is a tree, it suffices to show that
\mbox{$\hat \tau{ - }\hat \iota$} is injective.
Using the natural left $F$-action   on \mbox{$\ZZ [\,\EE T\mkern1mu]$},
we may form   the semi-direct-product group \mbox{$(\begin{smallmatrix}
 F&\ZZ [\,\EE T\mkern1mu]\\
\{0\} & \{1\}
\end{smallmatrix})$} with  matrix-style multiplication,  wherein  each element
\mbox{$(\begin{smallmatrix}
a&b \\
  0  &  1
\end{smallmatrix})$}  is \vspace{.5mm} denoted
\mbox{$\lceil a, b \rceil$}.
 Since  $Y$ is an  $F$-basis, there exists a unique group morphism
\mbox{$  F \to (\begin{smallmatrix}
 F&\ZZ [\,\EE T\mkern1mu]\\
\{0\} & \{1\}
\end{smallmatrix})$},
\mbox{$f \mapsto \lceil \phi f, \alpha f \rceil$},
such that
\mbox{$ \lceil \phi y ,\alpha y  \rceil  =   \lceil y,  1{\otimes} y\rceil$} for each \mbox{$y \in Y$}.
For all \mbox{$f, \, g  \in F$}, \vspace{-3mm}
$$  \lceil \phi(f g) ,\alpha(f g)\rceil
=   \lceil\phi f  , \alpha f  \rceil \lceil \phi g   , \alpha g  \rceil
= \lceil (\phi f)  (\phi g),  (\phi f) (\alpha g) + \alpha  f \rceil.\vspace{-1mm}$$
Then  \mbox{$\phi:F \to F$} is the identity map,  since \mbox{$\phi y   =  y$} and
\mbox{$\phi(f g)  =   (\phi f) (\phi g)$}.
The map \mbox{$\alpha:F \to \ZZ [\,\EE T\mkern1mu]$} satisfies \mbox{$\alpha y  =  1{\otimes} y$}
and  \mbox{$\alpha(f g) = (\phi f)(\alpha g) +    \alpha f$}.  Thus, we have
 a map  \mbox{$\alpha:  \VV\mkern-1mu T  \to \ZZ [\,\EE T\mkern1mu]$}
such that, for each \mbox{$v{\otimes} y \in  \EE T$},\vspace{-1mm}
   $$\alpha\bigl(  \tau(v{\otimes} y)\bigr){-} \alpha\bigl( \iota(v{\otimes} y)\bigr)
 =\alpha(v y){ - }\alpha(v) = (\phi v) (\alpha y)= (v)(1{\otimes} y)= v{\otimes} y.\vspace{-1mm}$$
Now $\alpha$  induces a
\mbox{$\ZZ\mkern2mu$}-module morphism \vspace{-1mm}
 \mbox{$\hat\alpha:  \ZZ[\,\VV\mkern-1mu T\mkern1mu]  \to \ZZ [\,\EE T\mkern1mu]$}, and
the composite\vspace{-1mm}
$$\ZZ [\,\EE T\mkern1mu] \xrightarrow{\hat \tau - \hat \iota}   \ZZ [\,\VV\mkern-1mu T\mkern1mu]
 \xrightarrow{\hat \alpha}  \ZZ [\,\EE T\mkern1mu]\vspace{-1mm}$$
is the identity map on \mbox{$\ZZ [\,\EE T\mkern1mu]$},
since it carries each \mbox{$v{\otimes} y \in  \EE T$} to itself.  Hence, \mbox{$\hat \tau{ - }\hat \iota$}
 is injective, as desired.
 \end{proof}

 \begin{definitions}\label{defs:height}
For each straight word $r$ in $F$, there exists some reduced \mbox{$Y^{\pm 1}$}-sequence
\mbox{$(y_1,y_2,\ldots,y_\ell)$} such that \mbox{$y_1 y_2\cdots y_\ell= r$}.  Here,\vspace{-1mm}
$$1 \xrightarrow{{\cdot}y_1}{\mkern-9mu-} y_1 \xrightarrow{{\cdot}y_2}{\mkern-9mu-} y_1y_2 \to{\mkern-9mu-}
\cdots \to{\mkern-9mu-} y_1y_2 \cdots y_{\ell-1} \xrightarrow{{\cdot}y_\ell}{\mkern-9mu-}   y_1 y_2\cdots y_\ell= r$$
is a reduced \mbox{$(F \rightwreath Y)$}-path from $1$ to $r$,
which is unique  by Theorem~\ref{th:tree}.
Thus, \mbox{$(y_1,y_2,\ldots,y_\ell)$} is unique, and
 we call it the reduced \mbox{$Y^{\pm 1}$}-sequence \textit{for}~$r$.
We set \mbox{$Y\mkern-2mu\text{-length}(r):= \ell$}  and
\mbox{$Y_{\vert y}\mkern1mu\text{-length}(r):=
 \bigl\vert\bigl\{i \in \{1,2,\ldots, \ell\} :  y_i \in \{y\}^{\pm 1}\bigr\}\bigr\vert$},
 for each    \mbox{$y \in Y^{\pm 1}$}.

For each cyclic word $r$ in $F$, there exists some cyclically reduced
\mbox{$Y^{\pm 1}$}-sequence \mbox{$(y_1, \ldots,y_\ell)$} such that
\mbox{$y_1 y_2\cdots y_\ell \in r$}.
 Here, \mbox{$(y_1, \ldots,y_\ell)$} is unique up to cyclic permutation,
as may be seen by considering the \mbox{${\cdot}Y$}-labelled
quotient graphs  \mbox{$\langle g \rangle \backslash (F\rightwreath Y)$}, \mbox{$g \in r$}, which are all
isomorphic.
We set \mbox{$Y\mkern-2mu\text{-length}(r):= \ell$} and
\mbox{$Y_{\vert y}\mkern1mu\text{-length}(r):=
 \bigl\vert\bigl\{i \in \{1,2,\ldots, \ell\} :  y_i \in \{y\}^{\pm 1}\bigr\}\bigr\vert$} for \mbox{$y \in Y^{\pm 1}$}.

Recall that $R$ is a finite set of words in $F$.
We set \mbox{$h(Y):=   \sum_{r \in R} Y\text{-length}(r)$} \vspace{.5mm} and
  \mbox{$h(Y_{\vert y}):=   \sum_{r \in R} Y_{\vert y}\mkern2mu\text{-length}(r)$}.
It is clear that
\mbox{$h(Y_{\vert y^{-1}}) =  h(Y_{\vert y})$}
and \mbox{$h(Y) =   \sum_{y \in Y} h(Y_{\vert y})$}. \qed
\end{definitions}

\section{Gersten's graphs and Whitehead's proof} \label{sec:gersten}

\begin{definitions}\label{defs:tangle}
Consider any   subset $V$ of $F$.
We let \mbox{$X \leftwreath V$},  \mbox{$V \rightwreath Y$}, and
 \mbox{$X \leftwreath V \rightwreath Y$}   denote the full subgraphs
 of \mbox{$X \leftwreath F$},  \mbox{$F \rightwreath Y$}, and
 \mbox{$X \leftwreath F \rightwreath Y$}with vertex-set $V$  respectively, where
 a subgraph $\Gamma_0$ of a graph
  $\Gamma$ is said to be \textit{full} if $\Gamma_0$~contains every
$\Gamma$-edge whose initial and terminal vertices lie in~$\Gamma_0$.
By Theorem~\ref{th:tree},  \mbox{$X \leftwreath F$}   and
 \mbox{$F \rightwreath Y$}  are trees; thus, \mbox{$X \leftwreath V$} and \,\mbox{$V\rightwreath Y$}
are forests. A subset of \mbox{$X \leftwreath F \rightwreath Y$} is said to be
  \textit{\mbox{$1$}-containing} it it contains~\mbox{$1$}.
We say that $V$ is an \textit{\mbox{$(X,Y)$}-translator} if
$V$  is a  \mbox{$1$}-containing  $F$-generating set such that
  \mbox{$X \leftwreath V$}  and  \mbox{$V\rightwreath Y$}
are trees.  In this event, we let \mbox{$(X \leftwreath V \rightwreath Y)_{\ge 3}$}
denote the set of elements of \mbox{$V{-}\{1\}$}
which have \mbox{$(X \leftwreath V \rightwreath Y)$}-valence at least~$3$.
Notice that \mbox{$\vert V{-}\{1\}\vert \ge \rank(F)$}, since
 \mbox{$V{-}\{1\}$} generates~$F$.

Clearly, $F$ itself is an \mbox{$(X,Y)$}-translator.  Let \mbox{$\kappa$}
denote the minimum value for \mbox{$\vert V{-}\{1\} \vert$}
as $V$ ranges over the set of all \mbox{$(X,Y)$}-trans\-la\-tors.
If  \mbox{$\kappa >  \rank(F)$}, we define   \mbox{$\operatorname{d}(X,Y):= \kappa$}.
Otherwise,   \mbox{$\kappa =  \rank(F)$}, and we then define  \mbox{$\operatorname{d}(X,Y)$}   to be the
minimum value  for \mbox{$\bigl\vert (X \leftwreath V \rightwreath Y)_{\ge 3} \bigr\vert$}
as $V$ ranges over the set of all \mbox{$(X,Y)$}-translators of cardinal \mbox{$1{+}\rank(F)$}. \qed
\end{definitions}

\begin{lemma}[Gersten(1987)]\label{lem:gerstenbis}  \mbox{$\operatorname{d}(X,Y) \in \naturals$}.
 \end{lemma}

\begin{proof}[Proof  \normalfont{(Krsti\'c(1989), here streamlined).}]
For each finite  \mbox{$1$}-containing \vspace{-.5mm} subset $W$ of $F$,
we let \mbox{$\breve{\mathbb{X}} W$} and \mbox{$\breve{\mathbb{Y}}  W $}
denote the vertex-sets  of the  \mbox{$1$}-containing components  of the forests~\mbox{$ X \leftwreath W$}
and \mbox{$W\rightwreath Y$} respectively; also, we let
\mbox{$\overline{\mathbb{X}} W $} and \mbox{$ \overline{\mathbb{Y}} W $}   denote the  vertex-sets  of the
tree-closures of~$W$ in the trees  \mbox{$X\leftwreath F$} and \mbox{$F\rightwreath Y$} respectively, where
  the \textit{tree-closure of~$W$} in a tree
is the  inclusion-smallest subtree which includes~$W$.
We  have now defined four self-maps of the set   of  finite  \mbox{$1$}-containing subsets of $F$.

Set \mbox{$\tilde Y := \{1\} \cup  Y^{\pm 1}$}
 and \mbox{$V :=    \breve{\mathbb{Y}} \overline{\mathbb{X}} \,
\overline{\mathbb{Y}}\, \overline{   \mathbb{X}}\,  \tilde Y$}.

Then $V$ is  a finite  \mbox{$1$}-containing subset of $F$, \mbox{$V \rightwreath Y$} is a tree, and

  \vspace{-6mm}

\begin{equation} \label{eq:inc}
 V  = \breve{\mathbb{Y}}  \bigl(\overline{\mathbb{X}} ( \overline{\mathbb{Y}} \,
\overline{\mathbb{X}}\,  \tilde Y)\bigr)
 \supseteq   \breve{\mathbb{Y}}\bigl(
  (\overline{\mathbb{Y}}\, \overline{ \mathbb{X}} \, \tilde Y)\bigr)
=   \overline{\mathbb{Y}} \,  \overline{\mathbb{X}}\,  \tilde Y  \supseteq
   \overline{\mathbb{X}}\,  \tilde Y    \supseteq \tilde Y.\vspace{-1.5mm}
\end{equation}
In particular,  $V$ is an  $F$-generating set.

We now prove that  \mbox{$ (\breve{\mathbb{X}}\mkern2mu V){\cdot} \tilde Y \subseteq
\breve{\mathbb{X}} \mkern 2mu ( V \mkern-2mu {\cdot} \tilde Y)$}.
Let \mbox{$y \in  \tilde Y$} and  \mbox{$v \in  \breve{\mathbb{X}}\mkern.5mu V$}; thus,
  \mbox{$V \supseteq \overline{\mathbb{X}}\mkern2mu\{v,1 \}$}.  Then
\mbox{$V \supseteq \overline{\mathbb{X}}\mkern2mu \{v,1, y^{-1}  \}$}, since
  \mbox{$V  \supseteq  \overline{\mathbb{X}}\, \tilde Y$},  by~\eqref{eq:inc}.
Now \mbox{$V{\cdot} \tilde Y  \supseteq (\overline{\mathbb{X}}\mkern2mu \{v,1,y^{-1}  \}) {\cdot} y =
\overline{\mathbb{X}}\mkern2mu \{ v{\cdot}y, y, 1\}$},
since \mbox{$X \leftwreath F$} is a right $F$-tree.
Thus, \mbox{$v{\cdot}\mkern1muy \in
\breve{\mathbb{X}} \mkern 2mu ( V \mkern-2mu {\cdot} \tilde Y)$}, as desired.

It follows from the definition of \mbox{$\breve {\mathbb{Y}}$} that
 $V$ is the inclusion\d1small\-est $1$-containing subset of~$F$  such that
\mbox{$\overline{ \mathbb{X}} \,\overline{\mathbb{Y}}\,\overline{  \mathbb{X}} \, \tilde Y
\cap  V \mkern-2mu{\cdot} \tilde Y    \,\subseteq\, V$}.  Now
\mbox{$\breve{\mathbb{X}}\mkern2mu V$} is a $1$-containing subset of $V$, and \vspace{-1mm}
$$\overline{ \mathbb{X}} \,\overline{\mathbb{Y}}\,\overline{  \mathbb{X}} \, \tilde Y
\cap (\breve{\mathbb{X}}\mkern2mu V) {\cdot} \tilde Y \,\,\subseteq\,\,
\breve{\mathbb{X}}(\overline{\mathbb{X}} \,\overline{\mathbb{Y}}\,\overline{\mathbb{X}} \, \tilde Y)
\cap  \breve{\mathbb{X}} \mkern 2mu ( V \mkern-2mu{\cdot} \tilde Y )  \,\,\subseteq\,\,
\breve{\mathbb{X}}(\overline{\mathbb{X}} \,\overline{\mathbb{Y}}\,\overline{  \mathbb{X}} \, \tilde Y
\cap     V \mkern-2mu{\cdot} \tilde Y)  \,\,\subseteq\,\, \breve{\mathbb{X}}(V).\vspace{-1.5mm}$$
It follows from the minimality property of $V$  that \mbox{$\breve{\mathbb{X}} V  = V$}. Thus,
\mbox{$X \leftwreath V$} is a tree.

 Hence, $V$ is a finite \mbox{$(X,Y)$}-translator.
\end{proof}

\begin{theorem}[Whitehead(1936b)]\label{th:main} With {\normalfont Notation~\ref{not:movesbis}}, if $X$ and $Y$ are
local-minimum points for $h$,   then either
  \mbox{$X^{\pm 1}  =   Y^{\pm 1}$}
or some  Whitehead transform $Y'$ of $Y$
 satisfies \mbox{$h(Y') = h(Y)$} and \mbox{$\operatorname{d}(X,Y') < \operatorname{d}(X,Y)$}.
\end{theorem}

\begin{proof}[Proof \normalfont{(Whitehead(1936b), here translated).}] For all \mbox{$ v, g  \in F$}, we let
\mbox{$v\overset{X: g{\cdot} }{\text{\textbf{--}}\cdots\mkern-2mu\to}{\mkern-10mu-} g{\cdot}v$}
denote\vspace{-.5mm} the unique reduced \mbox{$(X\leftwreath F)$}-path from $v$ to~$g{\cdot}v $,
and \mbox{$ v\overset{ {\cdot}g\mkern2mu : Y}{\text{\textbf{--}}\cdots\mkern-2mu\to}{\mkern-10mu-} v{\cdot}g $}
denote the unique reduced \mbox{$(F\rightwreath Y)$}-path from $v$ to $v{\cdot}g$.
If \mbox{$g=1$}, then these paths have length zero.

We shall obtain information about
   Whitehead transforms of $Y$ that are constructed using
  a procedure that depends  on
 \mbox{$\operatorname{d}(X,Y)$}.
We begin by describing  features that apply whenever we have an   \mbox{$(X,Y)$}-translator~$V$.

For each \mbox{$x \in X^{\pm 1}$}, we set \mbox{$\hat\iota_X x := x^{-1}{\cdot} V \cap V$} and
\mbox{$\hat \tau_Xx:= x {\cdot} V \cap V =  x {\cdot}  \hat \iota_X$}.
For each \mbox{$y \in Y^{\pm 1}$}, we set \mbox{$\hat\iota_Y y :=  V {\cdot}y^{-1} \cap V$} and
\mbox{$\hat \tau_Yy:= V  {\cdot} y   \cap V = \hat \iota_Y {\cdot} y $}.

Consider any \mbox{$y \in Y^{\pm 1}$}.  We shall now show that
  \mbox{$X \leftwreath (\hat \iota_Y y)$} and \mbox{$X \leftwreath (\hat \tau_Y y)$}
are subtrees of the tree  \mbox{$X \leftwreath V$}, and that
\mbox{$\bigl(X \leftwreath (\hat \iota_Y y)\bigr){\cdot} y =  X \leftwreath (\hat \tau_Y y).$}
We first show that \mbox{$\hat\iota_Y y \ne \emptyset$}.
  Since $V$  generates~$F$,
there exists some
\mbox{$u \in V {-} \langle Y {-}\{y\}^{\pm 1}\rangle$}.
Let \mbox{$(y_1, y_2, \ldots, y_\ell)$} be the reduced \mbox{$Y^{\pm 1}$}-sequence for $u$; thus,
 there exists some \mbox{$k \in \{1,2, \ldots, \ell\}$} such that \mbox{$\{y_k\}^{\pm 1} = \{y\}^{\pm 1}$}.
The reduced \mbox{$(F\rightwreath Y)$}-path from $1$ to $u$ is \nopagebreak then\vspace{-1.5mm}
 $$ 1 = u_0 \xrightarrow{{\cdot}y_1}{\mkern-9mu-} u_1 \xrightarrow{{\cdot}y_2}{\mkern-9mu-} u_2 \cdots
 \xrightarrow{{\cdot}y_\ell}{\mkern-9mu-} u_\ell = u; \vspace{-1.5mm}$$
 this is a \mbox{$(V\rightwreath Y)$}-path,   since the endpoints lie in $V$,
and the subpath  \mbox{$u_{k-1} \xrightarrow{{\cdot}y_k}{\mkern-9mu-} u_k $}
meets  \mbox{$\hat\iota_Y y$}, as desired.
Now consider any \mbox{$v, w \in \hat\iota_Yy$}.
Then \mbox{$v{\cdot}y, w{\cdot}y \in \hat \tau_Yy$}.
Let \mbox{$(x_\ell, x_{\ell-1}, \ldots, x_1)$} be the reduced \mbox{$X^{\pm 1}$}-sequence for
  \mbox{$  w{\cdot}v^{-1}=  ( w{\cdot}y){\cdot} ( v{\cdot}y)^{-1}$}.
The reduced \mbox{$(X \leftwreath F)$}-paths\vspace{-1.5mm}
$$   v = v_0 \xrightarrow{x_1{\cdot}}{\mkern-9mu-} v_1 \cdots \xrightarrow{x_\ell{\cdot}}{\mkern-9mu-} v_\ell = w
 \quad \text{and} \quad   v{\cdot}y = v_0{\cdot}y \xrightarrow{x_1{\cdot}}{\mkern-9mu-}
v_1{\cdot}y \cdots \xrightarrow{x_\ell{\cdot}}{\mkern-9mu-} v_\ell{\cdot}y = w{\cdot}y\vspace{-1mm}$$
are  \mbox{$(X \leftwreath V)$}-paths,  since their endpoints lie in $V$. Thus,
\mbox{$\{v_0, v_1, \ldots, v_\ell\}{\cdot}\{1,y \} \subseteq V.$}
This proves that \mbox{$X \leftwreath (\hat \iota_Y y)$} is a subtree of the tree  \mbox{$X \leftwreath V$}.
Also, \mbox{$\bigl(X \leftwreath (\hat \iota_Y y)\bigr){\cdot} y =  X \leftwreath (\hat \tau_Y y),$}
and   \mbox{$X \leftwreath (\hat \tau_Y y)$} is a subtree of the tree  \mbox{$X \leftwreath V$}.

Analogous assertions hold for \mbox{$(\hat \iota_X x)  \rightwreath Y$}
and  \mbox{$(\hat \tau_X x)  \rightwreath Y$}.

Consider any   \mbox{$v,w \in V$}  and
 any  \mbox{$(X\leftwreath V \rightwreath Y)$}-path $p$ from $v$ to $w$.
Let \mbox{$(x_1{\cdot}, x_2{\cdot}, \ldots, x_\ell{\cdot})$} be the
sequence of \mbox{$X^{\pm 1}{\cdot}$}\,-labels encountered along $p$.  We call the
\mbox{$X^{\pm 1}$}-sequence  \mbox{$(x_\ell, \ldots, x_2,x_1)$}  the  \textit{left \mbox{$X^{\pm1}$}-label of $p$},
and call \mbox{$g :=x_\ell \cdots x_2 x_1$}  the \textit{left $F$-label of $p$}.
Let \mbox{$({\cdot}y_1, {\cdot}y_2 , \ldots, {\cdot}y_{\ell'})$} be the
sequence of \mbox{${\cdot}Y^{\pm 1}$}\,-labels encountered along $p$.  We call the
\mbox{$Y^{\pm 1}$}-sequence  \mbox{$(y_1, y_2, \ldots, y_{\ell'})$}  the \textit{right \mbox{$Y^{\pm1}$}-label of $p$},
and call \mbox{$g':=y_1 y_2 \cdots y_{\ell'}$}   the \textit{right $F$-label of $p$}.
It is not difficult to see that \mbox{$g   v g' = w$} in $F$.
We may use ordinary path reductions and assume that $p$ is a reduced
\mbox{$(X\leftwreath V \rightwreath Y)$}-path without changing  the  left and right $F$-labels.
If the right \mbox{$Y^{\pm1}$}-label of $p$ is still not a reduced $Y^{\pm 1}$-sequence, then
 $p$ has some subpath $p'$ of the form\vspace{-2mm}
$$
u \xrightarrow{ {\cdot}y}{\mkern-9mu-} u{\cdot} y
\overset{  X: h   {\cdot} }{\text{\textbf{--}}\cdots\mkern-2mu\to}{\mkern-9mu-} h{\cdot} u
  {\cdot} y \xrightarrow{{\cdot}  y^{-1} } {\mkern-9mu-} h   {\cdot} u,\vspace{-1mm}
$$
 for some  \mbox{$h \in F{-}\{1\}$}.  Since \mbox{$X \leftwreath V$} is a tree,
we have the \mbox{$(X\leftwreath V)$}-path $p''$ which is
\mbox{$u\overset{X : h{\cdot} }{\text{\textbf{--}}\cdots\mkern-2mu\to} h  {\cdot} u
$}.  The \mbox{$(X\leftwreath V \rightwreath Y)$}-path obtained from $p$ by
 replacing $p'$ with $p''$
is said to be a  \textit{right $Y$-reduction of~$p$}.
This gives a shorter \mbox{$(X\leftwreath V \rightwreath Y)$}-path  from $v$ to $w$
with the same left and right $F$-labels,
the same left \mbox{$X^{\pm1}$}-label, and a  shorter right \mbox{$Y^{\pm1}$}-label.  Similar considerations give
 \textit{left $X\mkern-4mu$-reductions of~$p$}.
Any \mbox{$(X\leftwreath V \rightwreath Y)$}-path yields an \mbox{$(X\leftwreath V \rightwreath Y)$}-path
with reduced left \mbox{$X^{\pm1}$}-~and right \mbox{$Y^{\pm1}$}-labels after applying ordinary,
 left $X$-, and right $Y$-reductions sufficiently often.

Similar considerations apply for \textit{cyclic} ordinary, left $X$-, and right $Y$-reductions of closed
\mbox{$(X\leftwreath V \rightwreath Y)$}-paths; these operations
may change where the path is based.

We write \mbox{$\text{Paths}(X\leftwreath V \rightwreath Y)$} to denote the set of all
\mbox{$(X\leftwreath V \rightwreath Y)$}-paths.
We  construct a  map   \mbox{$F \to \operatorname{Paths}(X\leftwreath V \rightwreath Y) $}
which assigns  to each   \mbox{$g\in F$}  a closed \mbox{$(X\leftwreath V \rightwreath Y)$}-path
based at $1$ whose left \mbox{$X^{\pm1}$}-label is the reduced \mbox{$X^{\pm 1}$}-sequence for~$g^{-1}$,
and whose right  \mbox{$Y^{\pm1}$}-label is the  reduced \mbox{$Y^{\pm 1}$}-sequence  for \mbox{$g$}.
One way to do this is first to choose,  for  each \mbox{$x \in X$},   some \mbox{$v_x \in \hat \iota_Xx$},
and then  the \mbox{$(X\leftwreath V \rightwreath Y)$}-path \vspace{-2mm}
$$  1
\overset{  {\cdot} x{\cdot} v_x : Y}{\text{\textbf{--}}\cdots\mkern-2mu\to}{\mkern-9mu-}
 x{\cdot} v_x
\xrightarrow {x^{-1} {\cdot} }{\mkern-9mu-}
 v_x
\overset{  {\cdot} v_x^{-1} : Y  }{\text{\textbf{--}}\cdots\mkern-2mu\to}{\mkern-9mu-}
1 $$
 has left \mbox{$X^{\pm1}$}-label~\mbox{$(x^{-1})$}, which is the reduced \mbox{$X^{\pm 1}$}-sequence for $x^{-1}$.
Using inversion and concatenation of paths, we may now
assign to each \mbox{$g \in F$} a closed \mbox{$(X\leftwreath V \rightwreath Y)$}-path
based at~$1$ whose left \mbox{$X^{\pm 1}$}-label is the  reduced $X^{\pm 1}$-sequence for $g^{-1}$.
The left $F$-label is then~$g^{-1}$, and the right $F$-label must then be \mbox{$g$}.
By applying right $Y$-reductions, we obtain a
closed \mbox{$(X\leftwreath V \rightwreath Y)$}-path
based at $1$ whose left  \mbox{$X^{\pm 1}$}-label is still the reduced \mbox{$X^{\pm 1}$}-sequence for $g^{-1}$,
whose right  \mbox{$Y^{\pm 1}$}-label is a  reduced \mbox{$Y^{\pm 1}$}-sequence,
 and whose right $F$-label is still~\mbox{$g$}.
We call this the \textit{chosen  \mbox{$(X\leftwreath V \rightwreath Y)$}-path
representing $g$}.
The   reduced
$Y^{\pm 1}$-sequence for~$g$ and the reverse of the reduced $X^{\pm 1}$-sequence for $g^{-1}$
 have been interlaced to form a closed \mbox{$(X\leftwreath V \rightwreath Y)$}-path based at $1$.
For our counting purposes, the reverse of the reduced $X^{\pm 1}$-sequence for $g^{-1}$ contains the same
information as the reduced $X^{\pm 1}$-sequence for $g$; previous authors
 amalgamated \mbox{$F \rightwreath (X^{-1})$} and \mbox{$F \rightwreath Y$}
over their vertex-sets via the inversion map on $F$.

We now  construct a map \mbox{$R \to \operatorname{Paths}(X\leftwreath V \rightwreath Y)$}.
We map each  straight word $r$ contained in $R$ to the
chosen  \mbox{$(X\leftwreath V \rightwreath Y)$}-path
representing $r$.  For each cyclic word $r$ contained in $R$, we choose
an element $g$ of $r$, and consider the chosen  \mbox{$(X\leftwreath V \rightwreath Y)$}-path
representing $g$, and apply cyclic ordinary, left $X$-, and right $Y$-reductions,
until we get a closed  \mbox{$(X\leftwreath V \rightwreath Y)$}-path
whose right \mbox{$Y^{\pm 1}$}-label is a  cyclically reduced $Y^{\pm 1}$-sequence
and whose left \mbox{$X^{\pm 1}$}-label is a cyclically reduced $X^{\pm 1}$-sequence;
then the right  $F$-label is a   conjugate of \mbox{$g$},
and the left $F$-label is a conjugate of $g^{-1}$.
We call this the \textit{chosen  \mbox{$(X\leftwreath V \rightwreath Y)$}-path
representing $r$}.

Our   map \mbox{$R \to \operatorname{Paths}(X\leftwreath V \rightwreath Y)$} gives
 $R$   the structure of a set of closed
\mbox{$(X\leftwreath V \rightwreath Y)$}-paths, with the straight words being based at $1$.
We may now speak of the number of times an element  of~$R$ traverses a given
 \mbox{$(X\leftwreath V \rightwreath Y)$}-edge~$e$, and by  summing over all
elements of~$R$, we  may speak of the number of times $R$ traverses~$e$, and denote the number
by~\mbox{$\widetilde{\,h}(e)$}.

For each length-one \mbox{$(X \leftwreath V)$}-path \mbox{$v \xrightarrow{x{\cdot}} {\mkern-9mu-} w$},
we let \mbox{$ v  \overset {x {\cdot}}{\rightleftharpoons }  w $} denote the
\mbox{$(X \leftwreath V)$}-edge it traverses, and set
  \mbox{$\widetilde{\,h}(v \xrightarrow{x{\cdot}} {\mkern-9mu-} w):=
 \widetilde{\,h}( v  \overset {x {\cdot}}{\rightleftharpoons }  w)$};
thus, \mbox{$ w  \overset {x^{-1} {\cdot}}{\rightleftharpoons }  v $}  equals
 \mbox{$ v  \overset {x {\cdot}}{\rightleftharpoons }  w $},
and \mbox{$\widetilde{\,h}(w \xrightarrow{x^{-1}{\cdot}} {\mkern-9mu-} v)$} equals
\mbox{$ \widetilde{\,h}(v \xrightarrow{x{\cdot}} {\mkern-9mu-} w)$}.
For any element \mbox{$x $} of  \mbox{$X^{\pm 1}$}, and subsets \mbox{$V_0$} and \mbox{$V_1$} of \mbox{$V$},
we set\vspace{-1mm} $$\textstyle\widetilde{\,h} (V_0 \xrightarrow{x{\cdot}}{\mkern-9mu-} V_1):=
\sum\limits_{v \in  V_0 \cap (x^{-1} {\cdot} V_1)}   \widetilde{\,h}( v \xrightarrow{x{\cdot}} {\mkern-9mu-} x{\cdot} v).
\vspace{-2mm}$$
Notice that
 \mbox{$\widetilde{\,h}(V \xrightarrow{x{\cdot}} {\mkern-9mu-} V)
= \widetilde{\,h}( \hat\iota_X x \xrightarrow{x{\cdot}}{\mkern-9mu-}  \hat\tau_X x) = h(X_{\vert x})$},
since the left \mbox{$X^{\pm 1}$}-labels of the chosen   \mbox{$(X\leftwreath V \rightwreath Y)$}-paths are
 reduced,  and cyclically reduced for cyclic words.

Analogous notation applies with $Y$ in place of $X$.

 For any \mbox{$ x_\ast  \in X^{\pm 1}$},     \mbox{$y_\ast  \in   Y^{\pm 1}$}, and
 \mbox{$v_\ast \in \hat \iota_X x_\ast$}, we   say that \mbox{$(v_\ast, x_\ast,y_\ast)$}
 is a \textit{first-stage triple}, and associate to it
all of the following data.

The \mbox{$(X \leftwreath V)$}-edge \mbox{$ v_\ast \overset {x_\ast{\cdot}}{\rightleftharpoons }
     x_\ast{\cdot} v_\ast  $}
is called the \textit{disconnecting} edge.
Let   $V_0$  denote the
vertex-set  of the $1$-containing component  of the forest
\mbox{$  (X \leftwreath V) -\{  v_\ast \overset {x_\ast{\cdot}}{\rightleftharpoons }      x_\ast{\cdot} v_\ast  \},$}
  and set \mbox{$V_1:= V {-} V_0$}, the vertex-set  of the other  component.
We let \mbox{$\chi:V \to \{0,1\}$} denote the characteristic function of $V_1$; thus,
  \mbox{$v \in V_{\chi(v)}$}  for each   \mbox{$v \in V$}.
We   define a map
\mbox{$\hat \chi: \hat \iota_Y (Y^{\pm 1}) \to \{0,1\}$} as follows.
For \mbox{$j \in \{1,2\}$}, let \mbox{$Y^{\pm 1}_{j\text{-part}}$}
denote the set of those \mbox{$y \in Y^{\pm 1}{-}\{y_\ast\}$}
such that $\chi$ restricted to \mbox{$\hat \iota_Y y $} takes exactly $j$ values.
For each \mbox{$y \in Y^{\pm 1}_{1\text{-part}}$},
$\chi$ restricted to \mbox{$\hat \iota_Y y $} takes exactly one value,
and we define
\mbox{$\hat\chi(\hat \iota_Y y )$} to be that value.
Let \mbox{$\chi_F:F \to \{0,1\}$}  denote the characteristic function  of the vertex-set of that component of
\mbox{$(X \leftwreath F)-  \{v_\ast \overset {x_\ast{\cdot}}{\rightleftharpoons } x_\ast{\cdot}v_\ast\}$}
which does \textit{not} contain~$1$; the restriction of $\chi_F$ to $V$ is then $\chi$.
For each \mbox{$y \in Y^{\pm 1}_{2\text{-part}}$}, we define
 \mbox{$\hat\chi(\hat \iota_Y y ):= \chi_F(v_\ast{\cdot}y ^{-1})$}.
To complete  the definition of the map \mbox{$\hat \chi:\hat\iota_Y(Y^{\pm 1}) \to \{0,1\}$},
we set \mbox{$\hat\chi(\hat\iota_Y y_\ast):=  1-\hat\chi(\hat \iota_Y y_\ast^{-1})$}.

Let \mbox{$y_\dag$}   denote the element of \mbox{$\{y_\ast\}^{\pm 1}$}
such that \mbox{$\hat\chi(\hat \iota_Y y_\dag)= 0$};
hence, \mbox{$\hat\chi(\hat \tau_Y y_\dag)= \hat\chi(\hat \iota_Y y_\dag^{-1}) = 1$}.
For each \mbox{$y \in Y^{\pm 1} {-}\{y_\dag\}^{\pm 1}$},
we set \mbox{$y':= y_\dag^{\, \hat\chi(\hat\iota_Y y)} {\cdot} y{\cdot} y_\dag^{-\hat\chi(\hat \tau_Y y)}$},
while, for each \mbox{$y \in   \{y_\dag\}^{\pm 1}$}, we set \mbox{$y':= y$}.
We then set \mbox{$Y':=  \{ y' \mid y \in Y \}$}.  Thus, \mbox{$Y'$} is a
 Whitehead transform of~\mbox{$Y$}.
Since \mbox{$Y$} is a local-minimum point for $h$,
\mbox{$h(Y) \le h(Y')$}.
It is not difficult to see from the definition of~$Y'$  that, for each \mbox{$y \in Y^{\pm 1}{-} \{y_\dag\}^{\pm 1}$},
\mbox{$h(Y'_{\vert y'}) \ge h(Y_{\vert y})$}.  Similarly,
\mbox{$h(Y_{\vert y}) \ge h(Y'_{\vert y'})$}, and, hence, equality holds.
Now\vspace{-1mm}
\begin{equation}\label{eq:grow}
0 \le h(Y') - h(Y)  =  h (Y'_{\vert y_\dag'} ) - h(Y_{\vert y_\dag}).\vspace{-1mm}
\end{equation}

We next define a  map
\mbox{$\xi: \operatorname{Paths}(X\leftwreath V \rightwreath Y)
\to \operatorname{Paths}(X\leftwreath F \rightwreath Y')$}.
It  suffices to define $\xi$ on $V$ and  on the set of  length-one
\mbox{$(X \leftwreath V \rightwreath Y)$}-paths, and then concatenate paths.

We  define $\xi$~on $V$ by \vspace{-1mm}
$$V = V_0 \cup V_1 \to V_0 \cup V_1{\cdot} y_\dag ^{-1}
 \subseteq F, \qquad v \mapsto \xi(v):=v{\cdot}y_\dag^{-\chi(v)}.\vspace{-1mm}$$

Consider  a length-one \mbox{$(X \leftwreath V \rightwreath Y)$}-path of the form
\mbox{$ v \xrightarrow{{\cdot}y}{\mkern-9mu-} w $}, \mbox{$y \in Y^{\pm 1}$}.
We define \mbox{$\xi(v \xrightarrow{{\cdot}y}{\mkern-9mu-} w)$} to be
 \mbox{$\xi(v)  \overset{{\cdot}\xi(v)^{-1}{\cdot} \xi(w):Y'}{\text{\textbf{-----}}
\cdots\cdot\cdot\mkern-2mu\to}{\mkern-10mu-}\xi(w)$}.
Notice that\vspace{-3mm}
$$y_\dag^{ \hat\chi(\hat\iota_Yy)}{\cdot}y{\cdot} y_\dag^{-\hat\chi (\hat\tau_Yy)} = y'^{\delta(y)} \text{ where }
 \delta(y) := \begin{cases}
1 &\text{if } y \in Y^{\pm 1} - \{y_\dag\}^{\pm1},\\
0 &\text{if }y \in   \{y_\dag\}^{\pm1}.
\end{cases}\vspace{-3mm}$$
As \mbox{$\xi(v) = v{\cdot} y_\dag^{-\chi(v)}$}  and\vspace{-1mm}
$$
\xi(w) = w{\cdot} y_\dag^{-\chi(w)}
= v{\cdot}y{\cdot} y_\dag^{-\chi(w)}
= \xi(v) {\cdot} y_\dag^{\chi(v)}{\cdot}y{\cdot} y_\dag^{-\chi(w)}
= \xi(v) {\cdot}  y_\dag^{\chi(v)- \hat\chi(\hat\iota_Yy)}{\cdot}
y'^{\delta(y)}{\cdot} y_\dag^{\hat\chi (\hat\tau_Yy)-\chi(w)},
$$
 \vspace{-5mm}
\begin{equation}\label{eq:yunexc}
 \xi(v \xrightarrow{ {\cdot}y}{\mkern-9mu-}  w) \text{ equals } \xi(v)
\overset{{\cdot} y_\dag\mkern-3mu'^{\chi(v)-\hat\chi(\hat\iota_Yy)}{\cdot}y'^{\delta(y)} {\cdot}
y_\dag\mkern-3mu'^{\hat\chi(\hat\tau_Yy)- \chi(w)}: Y'}{\text{\textbf{--------}}
\cdots\cdots\cdots\cdots\cdots\cdots\cdot\cdot\mkern-2mu\to}{\mkern-10mu-}
 \xi(w).\phantom{----------}
\end{equation}

Consider now a length-one \mbox{$(X \leftwreath V \rightwreath Y)$}-path of the form
 \mbox{$ v \xrightarrow{x{\cdot}}{\mkern-9mu-} w $}, \mbox{$x \in X^{\pm 1}$}.
Here,  \vspace{-1mm}
$$\xi(v)  =  v {\cdot} y_\dag^{-\chi(v)} \text { and } \xi(w) = w {\cdot} y_\dag^{-\chi(w)}
= x {\cdot} v {\cdot} y_\dag^{-\chi(w)} =
x {\cdot} \xi(v) {\cdot} y_\dag^{\chi(v)-\chi(w)}.
$$

\vspace{-7mm}

\begin{align}\label{eq:xunexc}
\text{We shall define }\xi(v  \xrightarrow{ x {\cdot}}{\mkern-9mu-} w) &\text{ to be }
 \xi(v) \xrightarrow{x{\cdot}}{\mkern-9mu-} x{\cdot}\xi(v)
\overset{{\cdot}y_\dag\mkern-3mu'^{ \chi(v)-\chi(w)} : Y'}{ \textbf{---}
 \cdots \cdots\mkern-2mu\to}{\mkern-12mu-}
\xi(w)
\\& \,\,\text{or} \,\,
 \xi(v )
\overset{{\cdot}y_\dag\mkern-3mu'^{ \chi(v)-\chi(w)} : Y'}{\text{\textbf{---}}
 \cdots \cdots\mkern-2mu\to}{\mkern-12mu-}
 \xi(v){\cdot} y_\dag\mkern-3mu'^{ \chi(v)-\chi(w)}
 \xrightarrow{x {\cdot}}{\mkern-9mu-}
\xi(w).\phantom{-----}\nonumber
\end{align}
Recall that \mbox{$\chi(v) = \chi(w)$} unless
\mbox{$ v  \overset {x {\cdot}}{\rightleftharpoons } w$}
is the disconnecting edge \mbox{$ v_\ast \overset {x_\ast{\cdot}}{\rightleftharpoons }
     x_\ast{\cdot} v_\ast  $}.
If \mbox{$\chi(v) = \chi(w)$}, then \mbox{$\xi(v  \xrightarrow{ x {\cdot}}{\mkern-9mu-} w)$}
equals \mbox{$\xi(v) \xrightarrow{x{\cdot}}{\mkern-9mu-} \xi(w)$}.  Later, we
shall  have enough information to choose between the two options and define
\mbox{$\xi(v_\ast  \xrightarrow{ x_\ast {\cdot}}{\mkern-9mu-} x_\ast{\cdot}v_\ast)$} precisely.

 Now $\xi$ will convert  \mbox{$ (X  \leftwreath V \rightwreath Y )$}-paths   into
 \mbox{$ (X \leftwreath F \rightwreath Y')$}-paths  without changing the left \mbox{$X^{\pm 1}$}-labels, and, hence,
without changing the left $F$-labels.
 Since \mbox{$\xi(1) = 1{\cdot}y_\dag^{-\chi(1)}=1$}, we see that $R$ is now represented by
closed \mbox{$ (X \leftwreath F \rightwreath Y')$}-paths.
We do not claim that the right \mbox{$Y'^{\pm 1}$}-labels are    reduced, but
the image of $R$ in \mbox{$\operatorname{Paths}(X \leftwreath F \rightwreath Y')$}
 does give  an upper bound for~\mbox{$h (Y'_{\vert  y_\dag'} )$}.   On carefully\vspace{-1mm}
considering~\eqref{eq:xunexc}  and~\eqref{eq:yunexc}, and noting   that  the \mbox{$y'^{\delta(y)}$}-terms
contribute no \mbox{$y_\dag'$}-terms, we see that\vspace{-2mm}
$$
 h (Y'_{\vert  y_\dag'} )
\le \widetilde{\,h} ( v_\ast \xrightarrow{ x_\ast{\cdot}}{\mkern-9mu-}  x_\ast{\cdot} v_\ast )
+ \hskip-4pt \textstyle\sum\limits_{y \in Y^{\pm 1}  }
 \widetilde{\,h}  (   V_{1-\hat\chi(\hat\iota_Yy)}  \xrightarrow{{\cdot}y}{\mkern-9mu-} V).\vspace{-3mm}$$
Since \mbox{$   h(Y_{\vert y_\dag}) =   h(Y_{\vert y_\ast})$},  we    see from~\eqref{eq:grow}
that \vspace{-1mm}
\begin{equation}\label{eq:keyineq}
0   \le    h(Y') - h(Y)  \le
\widetilde{\,h}(v_\ast \xrightarrow{ x_\ast{\cdot}}{\mkern-9mu-}  x_\ast{\cdot} v_\ast )
-     h(Y_{\vert y_\ast})
  +  \hskip-2pt   \textstyle\sum\limits_{y \in Y^{\pm 1}}
\widetilde{\,h}  (    V_{1-\hat\chi(\hat\iota_Yy)}  \xrightarrow{{\cdot}y}{\mkern-9mu-}V).
\end{equation}

\vspace{-3.5mm}

We now consider  two cases.

\medskip

\noindent \textbf{Case 1:} \mbox{$\operatorname{d}(X,Y) \le \rank F$}.

Here, we assume that  \mbox{$\vert V{-}\{1\} \vert =   \rank F$} and
 \mbox{$ \bigl\vert(X \leftwreath V \rightwreath Y)_{\ge 3}\bigr\vert  = \operatorname{d}(X,Y)$}.

 Since \mbox{$V \rightwreath Y$} is a tree, we have\vspace{-1.2mm}
\mbox{$\sum\limits_{y \in Y} \vert \hat\iota_Yy\vert = \vert \operatorname{E}(V \rightwreath Y) \vert
= \vert V \vert {-} 1 = \vert Y \vert$}.  For each \mbox{$y \in Y$},
\mbox{$\vert \hat\iota_Yy \vert \ge 1$}; hence,  \mbox{$\vert \hat\iota_Yy \vert = 1$}.
Here in Case 1, for each \mbox{$y \in Y^{\pm 1}$},
we  write \mbox{$\iota_Y y$} to denote the unique element of~\mbox{$\hat \iota_Y y$},
and similarly for \mbox{$  \tau_Y y$}, and analogously with $X$ in place of~$Y$.

 As  an abelian group,
\mbox{$F/ [F,F]$} is freely generated  by the image  of any $F$-basis.  Hence,
there exists a unique\vspace{-1mm} map \mbox{$n_{X,Y} : X \times Y \to \integers$}, \mbox{$(x,y) \mapsto n_{x,y}$},
such that, for each \mbox{$y \in Y$},  $$\textstyle y  {\cdot}  [F,F]  =
\prod_{x \in X}  \bigl((x  {\cdot}
  [F,F])^{n_{x,y}}\bigr) \text{ in }  F / [F,F];$$
 we set \mbox{$X\mkern-6mu\operatorname{-absupp} (y):= \{x \in X \mid n_{x,y} \ne 0\}$}.
By choosing bijections from \mbox{$\{1,2,\ldots,\rank F\}$} to $X$ and to~$Y$,
we may view the map \mbox{$n_{X,Y}$} as an invertible matrix over $\integers$,
and view every bijection
 \mbox{$\phi: X \xrightarrow{\sim} Y$}, \mbox{$x \mapsto \phi x $}, as
a permutation  of \mbox{$\{1,2,\ldots,\rank F\}$}.  Then\nopagebreak \vspace{0mm}
$$\textstyle \sum\limits_{\phi:X \xrightarrow{\sim} Y}
( \operatorname{sign}(\phi){\cdot} \prod_{x \in X} n_{x, \phi x })
 = \operatorname{Det}(n_{X,Y}) \in \{1,-1\}.\vspace{-2.5mm}$$
There thus exists some bijection  \mbox{$\psi:X \xrightarrow{\sim} Y$} such that \vspace{.5mm}
 \mbox{$\prod_{x \in X} n_{x, \psi x }  \ne 0$};  we fix such a $\psi$ throughout Case 1.  Hence,
 \mbox{$x \in X\mkern-6mu\operatorname{-absupp} (\psi x )$}
for each \mbox{$x \in X$}.

Consider any  \mbox{$x_\ast \in X$}, and set \mbox{$y_\ast:= \psi(x_\ast) \in Y$}
and  \mbox{$v_\ast :=  \iota_Xx_\ast \in V$}.  We say that
\mbox{$(v_\ast, x_\ast,y_\ast)$}  is a   \textit{second-stage Case~$1$ triple}.
We have all the  data associated to a first-stage triple.

Let us first show that, for each \mbox{$y \in Y^{\pm1}$}, \mbox{$\hat\chi(\hat \iota_Yy)= \chi(\iota_Y y)$}.
Clearly  \mbox{$Y^{\pm 1}_{2\text{-part}} = \emptyset$}; hence, if
 \mbox{$y \in Y^{\pm1} {-}\{y_\ast\} = Y^{\pm1}_{1\text{-part}}$},
then \mbox{$\hat\chi( \hat \iota_Y y)   = \chi( \iota_y y)$}, as desired.
It remains to consider \mbox{$y_\ast$}.
Now\vspace{-1mm}
  $$\hat\chi( \hat \iota_Y y_\ast ) = 1- \hat\chi( \hat\iota_Y y_\ast^{-1} ) = 1- \chi( \iota_Y y_\ast^{-1})
= 1 - \chi( \tau_Y y_\ast),\vspace{-1mm}$$
and it suffices to show that \mbox{$\chi( \tau_Y y_\ast) \ne \chi(\iota_Y y_\ast)$}.
 Let \mbox{$(x_\ell, \ldots, x_2,x_1 )$}, \mbox{$\ell \in \naturals$}, be the reduced \mbox{$X^{\pm 1}$}-sequence for
\mbox{$( \tau_Y y_\ast){\cdot}( \iota_Y y_\ast)^{-1}$}.
Then\vspace{-1mm}
 $$ \iota_Y y_\ast {\cdot} y_\ast  {\cdot}( \iota_Y y_\ast)^{-1} =
 ( \tau_Y y_\ast){\cdot}( \iota_Y y_\ast)^{-1} =
 x_\ell \cdots x_2 x_1 .$$
Hence, \mbox{$ y_\ast{\cdot} [F,F]
 = \textstyle \prod_{k=1}^\ell (x_k {\cdot}  [F,F])$}.\vspace{1mm}
Since  \mbox{$x_\ast \in X\mkern-6mu\operatorname{-absupp} (\psi(x_\ast))$}
and \mbox{$ \psi(x_\ast)  =  y_\ast $},
there exists some  \mbox{$k \in \{1,2,\ldots,\ell\}$}
such that \mbox{$\{x_k\}^{\pm 1} = \{x_\ast\}^{\pm1}$}.
The reduced  \mbox{$(X \leftwreath F)$}-path
$$   \iota_Y y_\ast = v_0 \xrightarrow{x_1{\cdot}}{\mkern-9mu-} v_1 \xrightarrow{x_2{\cdot}}{\mkern-9mu-} \cdots
\xrightarrow{x_{\ell-1}{\cdot}}{\mkern-9mu-} v_{\ell-1}\xrightarrow{x_\ell{\cdot}}{\mkern-9mu-} v_\ell = x_\ell \cdots x_1 {\cdot}   \iota_Y y_\ast
 =   \tau_Y y_\ast $$
is  the unique reduced \mbox{$(X \leftwreath V)$}-path
from \mbox{$  \iota_Y y_\ast$} to \mbox{$ \tau_Y y_\ast $}, and it traverses
\mbox{$ v_{k-1}\overset {x_k{\cdot}}{\rightleftharpoons }     v_k  $},
which is
\mbox{$ v_\ast \overset {x_\ast{\cdot}}{\rightleftharpoons }     x_\ast{\cdot} v_\ast,$}
which is the disconnecting edge.
Hence,   \mbox{$\chi(  \iota_Yy_\ast) \ne \chi( \tau_Y y_\ast)$}, as desired.

Since \mbox{$y_\ast = \psi(x_\ast)$}, here in Case 1,~\eqref{eq:keyineq} takes the form \vspace{-1mm}
\begin{equation}\label{eq:Ps}
0  \le   h(Y') {-}  h(Y )
  \le     h(X_{\vert x_\ast}) { -}   h(Y_{\vert \psi(x_\ast)} ).\vspace{-1mm}
\end{equation}
Since $x_\ast$ is arbitrary,
\mbox{$0  \le   h(X_{\vert x}) { -}   h(Y_{\vert \psi x } )$}
for each \mbox{$x \in X$}.
Thus,\vspace{-1mm}  $$\textstyle
0 \le
 \sum\limits_{x \in X} \bigl(  h(X_{\vert x}) {-}   h(Y_{\vert \psi x } ) \bigr)
=
h(X)-h(Y). \vspace{-2mm}$$
By the interchangeability of $X$ and $Y$, we  then have
 \mbox{$h(X)- h(Y) = 0$}.  It  follows in turn that
\mbox{$ h(X_{\vert x }) { -}   h(Y_{\vert \psi x } ) = 0$}
 for each \mbox{$x \in X$}.
By~\eqref{eq:Ps}, \mbox{$h(Y') = h(Y)$}, as desired.

Consider the subcase where,  for each \mbox{$y \in Y^{\pm 1}$} such that \mbox{$ \iota_Y y = 1$},
the element \mbox{$ \tau_Y y $} of $V$ has
\mbox{$(X\leftwreath V \rightwreath Y)$}-valence
exactly two, and therefore   \mbox{$(X \leftwreath V)$}-valence exactly
one and \mbox{$(V \rightwreath Y)$}-valence exactly one.
The latter means that \mbox{$V^{\pm 1}{-}\{1\} = Y^{\pm 1}$}, and the former then means that
\mbox{$ V^{\pm 1}{-}\{1\} = X^{\pm 1}$}.  We then have
\mbox{$X^{\pm 1} = Y^{\pm 1}$}, which is one of the desired conclusions; here, \mbox{$\operatorname{d}(X,Y)=0$}.
It remains to consider the subcase where there exists some \mbox{$y_\dag \in
Y^{\pm 1}$} such that \mbox{$ \iota_Y y_\dag = 1$} and
 \mbox{$ \tau_Y y_\dag  \in (X \leftwreath V \rightwreath Y)_{\ge 3}$}.
 We fix such a~\mbox{$y_\dag$}, and take \mbox{$y_\ast \in Y \cap \{y_\dag\}^{\pm 1}$},
  \mbox{$x_\ast := \psi^{-1}(y_\ast)$}, and \mbox{$v_\ast :=  \iota_Xx_\ast$}.
We  say that \mbox{$(v_\ast, x_\ast, y_\ast)$} is a     \textit{third-stage Case~$1$ triple}.

By \eqref{eq:yunexc}, for each
 \mbox{$y \in Y^{\pm1}-\{y_\dag\}^{\pm 1}$},  we have
\mbox{$\xi ( \iota_Yy \xrightarrow{ {\cdot}y} {\mkern-9mu-}  \tau_Yy )$}
 equals   \mbox{$\xi( \iota_Yy)
 \xrightarrow{ {\cdot}y'}{\mkern-9mu-}\xi( \tau_Yy)$},
while
 \mbox{$\xi(\iota_Y y_\dag  \xrightarrow{ {\cdot}y_\dag} {\mkern-9mu-}  \tau_Y y_\dag)$}   equals
   \mbox{$\xi( \iota_Y y_\dag) \overset{{\cdot}1:Y'}{\text{\textbf{--}}\cdots\mkern-2mu\to}{\mkern-10mu-}\xi( \tau_Y y_\dag)$}.

Set   \mbox{$x_\dag:= x_\ast^{(-1)^{\chi(v_\ast)}}$};
thus,
\mbox{$  \iota_X x_\dag \overset {x_\dag{\cdot}}{\rightleftharpoons }      \tau_X x_\dag$}    equals
 \mbox{$v_\ast \overset {x_\ast{\cdot}}{\rightleftharpoons }      x_\ast{\cdot} v_\ast$},
\mbox{$\chi(\iota_X x_\dag) = 0$},  and  \mbox{$\chi(\tau_X x_\dag) = 1.$}
In~\eqref{eq:xunexc},   for  each  \mbox{$x \in X^{\pm1}-\{x_\dag\}^{\pm 1}$},
\mbox{$\xi ( \iota_Xx \xrightarrow{x{\cdot}}{\mkern-9mu-}    \tau_Xx  )$}
equals \vspace{-.5mm} \mbox{$\xi( \iota_Xx)  \xrightarrow{x{\cdot}}{\mkern-9mu-}\xi( \tau_Xx)$},
while we now choose \mbox{$
\xi( \iota_X x_\dag  \xrightarrow{ x_\dag{\cdot}}{\mkern-9mu-}   \tau_X x_\dag)$}   to be equal to
 \mbox{$\xi( \iota_X x_\dag) \xrightarrow{x_\dag {\cdot}}{\mkern-9mu-}  \tau_X x_\dag \xrightarrow{ {\cdot}y_\dag'^{-1}  }{\mkern-9mu-}
\xi( \tau_X x_\dag)$}.

Set \mbox{$V':= \xi(V)  \cup \hat \tau_X  x_\dag  = V_0 \cup V_1{\cdot} y_\dag^{-1} \cup
    \hat \tau_X  x_\dag    \subseteq F$}.
We shall see that  \mbox{$V'$} is an \mbox{$(X, Y')$}-trans\-lator.
Since \mbox{$\hat\tau_Xx_\dag \subseteq V_1$}, we see that \mbox{$V'$} is a
finite, $1$-containing, $F$-generating set.  Thus, \mbox{$\vert V' \vert \ge    \vert V\vert$}.
Since \mbox{$\hat \iota_Y y_\dag  \subseteq  V_0 \cap V_1{\cdot} y_\dag^{-1}$},
we see that \vspace{.5mm} \mbox{$\vert V' \vert = \vert V\vert$} and
\mbox{$ \tau_X  x_\dag  \not \in V_0 \cup V_1{\cdot} y_\dag^{-1}$}.

It is clear that \mbox{$\xi\bigl(\operatorname{Paths}(X\leftwreath V \rightwreath Y)\bigr) \subseteq
\operatorname{Paths}(X\leftwreath V' \rightwreath Y')$}. \hskip-.8pt
Let us examine the graphs   \mbox{$ X \leftwreath V' \rightwreath Y' $},
\mbox{$ X \leftwreath V'  $}, and \mbox{$V' \rightwreath Y'$}.
From the form that $\xi$ takes here,\vspace{-1mm} we see that
 \mbox{$X \leftwreath V' \rightwreath Y' $} is obtained from
\mbox{$X \leftwreath V \rightwreath Y$} by  first
subdividing the edge \mbox{$ \iota_X x_\dag \overset {x_\dag{\cdot}}{\rightleftharpoons }      \tau_X x_\dag$},
and secondly \vspace{-1mm}  collapsing the edge
\mbox{$\iota_Y y_\dag  \overset {{\cdot} y_\dag}{\rightleftharpoons }  \tau_Y y_\dag $}.
The graph \mbox{$X \leftwreath V'$} is thus obtained from the tree
 \mbox{$X  \leftwreath V $} by first removing an edge,
leaving two components with vertex-sets $V_0$ and~$V_1$,
secondly  identifying  one vertex of $V_0$ with one vertex of $V_1$,
and thirdly attaching one new vertex and one new edge
incident to the new vertex and an old vertex.
Hence, \mbox{$X \leftwreath V'$} is  a tree.
The graph \mbox{$ V' \rightwreath Y' $} is obtained
from the tree \mbox{$ V \rightwreath Y  $} by first collapsing one edge identifying its vertices,
and secondly  attaching one new vertex and one new edge
incident to the new vertex and an old vertex.  Thus, \mbox{$ V' \rightwreath Y' $} is
 a tree.
Hence, \mbox{$V'$} is an \mbox{$(X, Y')$}-translator.

Finally,
 \mbox{$\bigl \vert (X \leftwreath V \rightwreath Y)_{\ge 3}\bigr\vert > \bigl\vert (X \leftwreath
V' \rightwreath Y')_{\ge 3}\bigr\vert$},
since the newly created vertex has \mbox{$(X \leftwreath V'\rightwreath Y')$}-va\-lence two,
while the two old vertices which become identified
are \mbox{$\tau_Yy_\dag \in (X \leftwreath V \rightwreath Y)_{\ge 3}$} and \mbox{$\iota_Yy_\dag = 1$}.
Hence, \mbox{$\operatorname{d}(X,Y) > \operatorname{d}(X,Y')$}.

\medskip

\noindent \textbf{Case 2:} \mbox{$ \operatorname{d}(X,Y) >  \rank F$}.

Here, we assume that \mbox{$  \vert V{-}\{1\} \vert =  \operatorname{d}(X,Y)$}.  Hence,
\mbox{$  \vert V{-}\{1\} \vert   >  \rank F = \vert Y \vert$}.

Since \mbox{$V \rightwreath Y$} is a tree,\vspace{-1.5mm}
\mbox{$\sum\limits_{y \in Y} \vert \hat\iota_Yy\vert = \vert \operatorname{E}(V \rightwreath Y) \vert
= \vert V \vert {-} 1 > \vert Y \vert$}.
There then exists some \mbox{$y_\ast \in Y^{\pm 1}$} such that \mbox{$\vert \hat \iota_Y y_\ast \vert \ge 2$}.
The tree \mbox{$X \leftwreath (\hat \iota_Y y_\ast)$} must then contain
 some edge \mbox{$ v_\ast \overset {x_\ast{\cdot}}{\rightleftharpoons } x_\ast{\cdot}v_\ast $},
and then the tree
\mbox{$  X \leftwreath (\hat \tau_Y y_\ast) = \bigl(X \leftwreath (\hat \iota_Y y_\ast)\bigr){\cdot}y_\ast$}
contains the edge
 \mbox{$ v_\ast{\cdot}y_\ast  \overset {x_\ast{\cdot}}{\rightleftharpoons } x_\ast{\cdot}v_\ast{\cdot}y_\ast $},
  giving a diagram \vspace{-4mm}
$$\begin{CD}
v_\ast{\cdot} y_\ast  @>{x_\ast {\cdot} }>>{\mkern-15mu-} x_\ast {\cdot}v_\ast {\cdot}y_\ast \\
@A{ {\cdot}y_\ast }AA @AA{{ \cdot}y_\ast}A\\
v_\ast @>{ x_\ast {\cdot} }>>{\mkern-30mu-} x_\ast{\cdot}v_\ast
\end{CD}\vspace{1mm}$$
 of length-one \mbox{$(X \leftwreath V \rightwreath Y)$}-paths.
We say that \mbox{$(v_\ast, x_\ast,y_\ast)$} is  a
\textit{second-stage Case~$2$ triple}.
We now have all the   data associated with a first-stage triple.

 If  \mbox{$y_\ast^{-1} \in Y_{1\text{-part}}$}, then
\mbox{$\hat\chi(\hat\iota y_{\ast}^{-1}) = \chi(v_\ast {\cdot}  y_\ast)$}, because
 \mbox{$ v_\ast {\cdot}  y_\ast \in  \hat\iota y_{\ast}^{-1} $}. If
\mbox{$y_\ast^{-1} \in Y^{\pm1}_{2\text{-part}}$}, then, by definition,
 \mbox{$\hat\chi(\hat\iota y_{\ast}^{-1}) = \chi_F(v_\ast {\cdot}  y_\ast)= \chi(v_\ast {\cdot}  y_\ast)$}.
This proves that  \mbox{$\hat\chi(\hat\iota y_{\ast}^{-1}) = \chi(v_\ast {\cdot}  y_\ast)$};
hence, \mbox{$\hat\chi(\hat\iota y_{\ast} ) = 1- \chi(v_\ast {\cdot}  y_\ast)$}.

Let \mbox{$\operatorname{west}_{(v_\ast, x_\ast )} (\hat\iota_Y y_\ast )$}
and \mbox{$\operatorname{east}_{(v_\ast, x_\ast )} (\hat\iota_Y y_\ast )$}
  denote the vertex-sets of  the components of
\mbox{$\bigl(X \leftwreath  (\hat \iota_Y y_\ast )\bigr)
\hskip-1pt -\hskip-1pt \{ v_\ast \overset {x_\ast{\cdot}}{\rightleftharpoons } x_\ast{\cdot}v_\ast  \}$}
which contain  \mbox{$v_\ast$} and \mbox{$x_\ast{\cdot} v_\ast$} respectively.
Let \mbox{$\operatorname{proper} (v_\ast, x_\ast, y_\ast)$}
 denote  the intersection of \mbox{$\hat \iota_Y y_\ast$} with the component of
\mbox{$\bigl(X \leftwreath  F\bigr)
\hskip-1pt -\hskip-1pt \{ v_\ast \overset {x_\ast{\cdot}}{\rightleftharpoons } x_\ast{\cdot}v_\ast  \}$}
which does \textit{not} contain \mbox{$v_\ast {\cdot} y_\ast$} and, hence,   intersects $V$ in
\mbox{$V_{1- \chi(v_\ast{\cdot}y_\ast)}$}. Since
 \mbox{$\hat\chi(\hat\iota y_{\ast} ) = 1- \chi(v_\ast {\cdot}  y_\ast)$},
  \mbox{$\operatorname{proper} (v_\ast, x_\ast, y_\ast)
= \hat \iota_Y y_\ast \cap  V_{\hat\chi(\hat\iota_Y y_\ast)} \in \{
\operatorname{west}_{(v_\ast, x_\ast )} (\hat\iota_Y y_\ast ),
 \operatorname{east}_{(v_\ast, x_\ast )} (\hat\iota_Y y_\ast )\}$}.

Let \mbox{$\operatorname{south}_{(v_\ast, y_\ast )} (\hat\iota_X x_\ast )$} and\vspace{-.5mm}
 \mbox{$\operatorname{north}_{(v_\ast, y_\ast )} (\hat\iota_X x_\ast ) $} denote the
vertex-sets of   the components of
\mbox{$\bigl((\hat\iota_X x_\ast )\rightwreath  Y\bigr)
\hskip-1pt -\hskip-1pt \{ v_\ast \overset {{\cdot}y_\ast}{\rightleftharpoons } v_\ast{\cdot}y_\ast  \}$}
which contain  \mbox{$v_\ast$} and \mbox{$ v_\ast{\cdot}y_\ast$} respectively. It is not difficult to show that
 \mbox{$\operatorname{south}_{(v_\ast, y_\ast )} (\hat\iota_X x_\ast ) = \{v_\ast\}$}
if and only if \mbox{$v_\ast$}
has \mbox{$(X\leftwreath V)$}-valence one,
if and only if \mbox{$Y^{\pm 1}_{2\text{-part}} =\emptyset$}.

We now consider an arbitrary \mbox{$y \in Y^{\pm1}_{2\text{-part}}$}.
Thus, \mbox{$\hat\chi(\hat \iota_Y y) = \chi_F(v_\ast{\cdot}y^{-1})$}.
  We
have a diagram
$$\begin{CD}
v_\ast{\cdot} y_\ast  @>{x_\ast {\cdot} }>>{\mkern-12mu-} x_\ast {\cdot}v_\ast {\cdot}y_\ast \\
@A{ {\cdot}y_\ast }AA @AA{{ \cdot}y_\ast}A\\
v_\ast @>{ x_\ast {\cdot} }>>{\mkern-35mu-} x_\ast{\cdot}v_\ast\\
@V{ {\cdot}y  }VV  @VV{{ \cdot}y }V\\
v_\ast{ \cdot}y   @>{ x_\ast {\cdot} }>>{\mkern-25mu-} x_\ast{\cdot}v_\ast{ \cdot}y
\end{CD}$$
 of length-one \mbox{$(X \leftwreath V \rightwreath Y)$}-paths.
Notice that \mbox{$(v_\ast {\cdot} y ,x_\ast, y ^{-1})$} is a   second-stage Case~2 triple, and that
 \mbox{$ \operatorname{proper} (v_\ast{\cdot} y  , x_\ast  , y^{-1})$}
 is then the intersection of \mbox{$\hat \iota_Y y^{-1}$} with that component of
\mbox{$\bigl(X \leftwreath  F\bigr)
 -\{  v_\ast {\cdot} y  \overset {x_\ast{\cdot}}{\rightleftharpoons } x_\ast{\cdot}v_\ast{\cdot} y  \}$}
which does \textit{not} contain \mbox{$v_\ast$}.  On right multiplying by~\mbox{$y^{-1}$}, we see that
 \mbox{$\bigl( \operatorname{proper} (v_\ast{\cdot} y , x_\ast , y^{-1})\bigr){\cdot}y ^{-1}$}
 is the intersection of \mbox{$\hat \iota_Y y$} with that component of
\mbox{$\bigl(X \leftwreath  F\bigr)
 -\{  v_\ast   \overset {x_\ast{\cdot}}{\rightleftharpoons } x_\ast{\cdot}v_\ast  \}$}
which does \textit{not} contain \mbox{$v_\ast{\cdot}y^{-1}$} and, hence,  intersects $V$ in
\mbox{$V_{1- \chi_F(v_\ast{\cdot}y ^{-1})}$}.  Since
\mbox{$ \hat\chi(\hat \iota_Y y)=\chi_F(v_\ast{\cdot}y ^{-1})$},
   \mbox{$\bigl (\operatorname{proper} (v_\ast{\cdot} y ,  x_\ast , y^{-1})\bigr){\cdot}y ^{-1}
= \hat \iota_Y y  \cap V_{1- \hat\chi(\hat \iota_Y y)}$}.

Now
\begin{align*}
&\textstyle\sum\limits_{y \in Y^{\pm 1}}
\widetilde{\,h}  (V_{1-\hat\chi(\hat\iota_Yy)}  \xrightarrow{{\cdot}y}{\mkern-9mu-}V) \\&\textstyle=
\widetilde{\,h}  (V_{1-\hat\chi(\hat\iota_Yy_\ast)}  \xrightarrow{{\cdot}y_\ast}{\mkern-9mu-}V)
+ \sum\limits_{y \in Y^{\pm 1}_{2\text{-part}}} \widetilde{\,h}  (V_{1-\hat\chi(\hat\iota_Yy  )}  \xrightarrow{{\cdot}y  }{\mkern-9mu-}V)
\\&\textstyle=   h(Y_{\vert y_\ast}) -
\widetilde{\,h}  (V_{ \hat\chi(\hat\iota_Yy_\ast)}  \xrightarrow{{\cdot}y_\ast}{\mkern-9mu-}V)
+ \sum\limits_{y \in Y^{\pm1}_{2\text{-part}}}
\widetilde{\,h} \Bigl(\bigl(\operatorname{proper} (v_\ast{\cdot}y , x_\ast, y^{-1})\bigr){\cdot}y^{-1}
\xrightarrow{{\cdot}y }{\mkern-9mu-} V\Bigr)
\\&\textstyle=  h(Y_{\vert y_\ast}) -
\widetilde{\,h} \bigl(\operatorname{proper} (v_\ast, x_\ast  , y_\ast)
\xrightarrow{{\cdot}y_\ast}{\mkern-9mu-} V\bigr)
+ \hskip-9pt\sum\limits_{y \in Y^{\pm 1}_{2\text{-part}}}  \widetilde{\,h}\bigl
(\operatorname{proper} (v_\ast{\cdot}y , x_\ast , y^{-1})
\xrightarrow{{\cdot}y ^{-1}}{\mkern-9mu-} V\bigr).
\end{align*}
Thus, here in Case~2,~\eqref{eq:keyineq} takes the form
\begin{align} \label{eq:trap}
 \textstyle 0 \le  h (Y'){-}  h(Y )
\textstyle\le\widetilde{\,h}( v_\ast \xrightarrow{ x_\ast{\cdot}}{\mkern-9mu-}  x_\ast{\cdot} v_\ast )
 & {-}\,
\widetilde{\,h} \bigl(\operatorname{proper} (v_\ast, x_\ast, y_\ast)
\xrightarrow{{\cdot}y_\ast}{\mkern-9mu-} V\bigr)
 \\&\textstyle+ \hskip-10pt\sum\limits_{y \in Y^{\pm 1}_{2\text{-part}}} \widetilde{\,h}\bigl
 (\operatorname{proper} (v_\ast{\cdot}y , x_\ast , y^{-1})
\xrightarrow{{\cdot}y^{-1} } {\mkern-9mu-}V\bigr).    \nonumber
\end{align}

\null\vspace{-7mm}

\noindent In particular,  \vspace{-1mm}
$$
\textstyle\widetilde{\,h} \bigl( \operatorname{proper} (v_\ast, x_\ast, y_\ast)
\xrightarrow{{\cdot}y_\ast}{\mkern-9mu-} V\bigr) \le
\widetilde{\,h}(v_\ast \xrightarrow{ x_\ast{\cdot}} {\mkern-9mu-}  x_\ast{\cdot} v_\ast ) + \hskip-10pt \textstyle\sum\limits_{y \in Y^{\pm 1}_{2\text{-part}}}
\widetilde{\,h}\bigl (\operatorname{proper} (v_\ast{\cdot}y, x_\ast, y^{-1})
\xrightarrow{{\cdot}y^{-1} } {\mkern-9mu-}V\bigr). \vspace{-3mm}
$$
Since $$  \textstyle\widetilde{\,h} \bigl( \operatorname{south}_{(v_\ast, y_\ast )} (\hat\iota_X x_\ast )
\xrightarrow{x_\ast{\cdot}}{\mkern-9mu-} V\bigr)  =
 \widetilde{\,h}(v_\ast \xrightarrow{ x_\ast{\cdot}}{\mkern-9mu-}  x_\ast{\cdot} v_\ast )
+ \hskip-10pt\sum\limits_{y \in Y^{\pm 1}_{2\text{-part}}}
\widetilde{\,h} \bigl( \operatorname{south}_{(v_\ast{\cdot}y, y^{-1})}
(\hat\iota_X x_\ast  )\xrightarrow{x_\ast{\cdot} } {\mkern-9mu-}V\bigr),\vspace{-1.5mm}$$
it may  be seen by induction on  \mbox{$\vert  \operatorname{south}_{(v_\ast, y_\ast )}
 (\hat\iota_X x_\ast ) \vert$} that
$$ \widetilde{\,h} \bigl( \operatorname{proper} (v_\ast, x_\ast  , y_\ast)
\xrightarrow{{\cdot}y_\ast}{\mkern-9mu-}V\bigr) \le
\widetilde{\,h} \bigl( \operatorname{south}_{(v_\ast, y_\ast )} (\hat\iota_X x_\ast )
\xrightarrow{x_\ast{\cdot}} {\mkern-9mu-}V\bigr).\vspace{-2mm} $$
Let us  write
$$\widetilde{\,h}\operatorname{-west}:=
\widetilde{\,h} \bigl( \operatorname{west}_{(v_\ast, x_\ast )} (\hat\iota_Y y_\ast )
\xrightarrow{{\cdot}y_\ast}{\mkern-9mu-}V\bigr) \text{ and }
 \widetilde{\,h}\operatorname{-south}:=\widetilde{\,h} \bigl( \operatorname{south}_{(v_\ast, y_\ast )} (\hat\iota_X x_\ast )
\xrightarrow{x_\ast{\cdot}}{\mkern-9mu-} V\bigr),$$ and similarly for \mbox{$\widetilde{\,h}\operatorname{-east}$}
and \mbox{$\widetilde{\,h}\operatorname{-north}$}.
We have shown that
 $$\min \{  \widetilde{\,h} \operatorname{-west}, \widetilde{\,h}\operatorname{-east}\}
 \le \widetilde{\,h} \bigl(\operatorname{proper}( v_\ast, x_\ast, y_\ast)
\xrightarrow{{\cdot}y_\ast}{\mkern-9mu-} V\bigr)
 \le \widetilde{\,h}\operatorname{-south}.$$
Replacing  \mbox{$(v_\ast, x_\ast, y_\ast)$}
 with the  second-stage Case 2 triple \mbox{$(v_\ast{\cdot}y_\ast, x_\ast, y_\ast^{-1})$}
 interchanges south and north, and we find that
 \mbox{$\min \{ \widetilde{\,h}\operatorname{-west}, \widetilde{\,h}\operatorname{-east}\}
\le \widetilde{\,h}\operatorname{-north}$}.
Hence,\vspace{-1mm} $$\min\{\widetilde{\,h}\operatorname{-west}, \widetilde{\,h}\operatorname{-east}\}
\le \min\{\widetilde{\,h}\operatorname{-south},\widetilde{\,h}\operatorname{-north}\}.\vspace{-1mm}$$
Interchanging $X$ and~$Y$ interchanges south and west, as well as north and east, and we find\vspace{-1mm}
\begin{equation}\label{eq:triple}
  \min\{\widetilde{\,h}\operatorname{-south},\widetilde{\,h}\operatorname{-north}\}\le
 \min\{\widetilde{\,h}\operatorname{-west}, \widetilde{\,h}\operatorname{-east}\}
 \le \widetilde{\,h} \bigl(\operatorname{proper}( v_\ast, x_\ast, y_\ast)
\xrightarrow{{\cdot}y_\ast}{\mkern-9mu-} V\bigr).\vspace{-1mm}
\end{equation}

We now choose a    third-stage Case $2$ triple  as follows.
Consider the preceding~\mbox{$x_\ast$}.  Thus,  \mbox{$(\hat\iota_X x_\ast) \rightwreath Y$} is a finite tree
that has at least one edge  and, hence,   at least two
valence-one vertices.
There then exists a valence-one
\mbox{$\bigl((\hat\iota_X x_\ast) \rightwreath Y\bigr)$}-vertex~\mbox{$v_\ast$} such that
\mbox{$\widetilde{\,h}( v_\ast \xrightarrow{x_\ast{\cdot}}{\mkern-9mu-}
 x_\ast{\cdot} v_\ast) \le     h (X_{\vert x_\ast})/2$}.
Taking \mbox{$ v_\ast \overset {x_\ast{\cdot}}{\rightleftharpoons }     x_\ast{\cdot} v_\ast  $}
as the disconnecting edge  determines
a map \mbox{$\chi:V \to \{0,1\}$}.
If  {$\chi(v_\ast) = 0$},
we fix this $x_\ast$ and this $v_\ast$.  If   \mbox{{$\chi(v_\ast) = 1$}, we replace $(x_\ast, v_\ast)$}
with  \mbox{$(x_\ast^{-1}, x_\ast{\cdot} v_\ast)$}, and then fix\vspace{-.5mm} this new $x_\ast$ and   $v_\ast$; then
\mbox{$\chi(v_\ast) = 0$}.  Now \mbox{$\chi(x_\ast{\cdot}v_\ast) = 1$}.
Let \mbox{$y_\ast$} denote the element of \mbox{$Y^{\pm 1}$}
such that  \mbox{$  v_\ast   \overset {{\cdot}y_\ast}{\rightleftharpoons }  v_\ast{\cdot} y_\ast    $}
 is the unique edge of \mbox{$(\hat\iota_X x_\ast) \rightwreath Y$} that is incident to~\mbox{$v_\ast$}.
Now \mbox{$(v_\ast, x_\ast, y_\ast)$} is a   second-stage Case~2 triple,
\mbox{$Y^{\pm 1}_{2\text{-part}} = \emptyset$},
\mbox{$ \widetilde{\,h}( v_\ast \xrightarrow{x_\ast{\cdot}}{\mkern-9mu-}
 x_\ast{\cdot} v_\ast) \le      h (X_{\vert x_\ast})/2$},
\mbox{$\chi(v_\ast) = 0$}, and \mbox{$\chi(x_\ast{\cdot}v_\ast) = 1$};
we say that \mbox{$(v_\ast, x_\ast, y_\ast)$} is a  \textit{third-stage Case $2$ triple}.
Since \mbox{$ Y^{\pm 1}_{2\text{-part}} = \emptyset$},\vspace{-4mm}
$$\widetilde{\,h}\operatorname{-south}  =
\widetilde{\,h}( v_\ast \xrightarrow{x_\ast{\cdot}}{\mkern-9mu-}
 x_\ast{\cdot} v_\ast) \le   h (X_{\vert x_\ast}) - \widetilde{\,h}( v_\ast \xrightarrow{x_\ast{\cdot}}{\mkern-9mu-}
 x_\ast{\cdot} v_\ast) = \widetilde{\,h}\operatorname{-north};\vspace{-1.5mm}$$
thus, \mbox{$\widetilde{\,h}( v_\ast \xrightarrow{x_\ast{\cdot}}{\mkern-9mu-}
 x_\ast{\cdot} v_\ast)  = \min\{ \widetilde{\,h}\operatorname{-south}, \widetilde{\,h}\operatorname{-north}\}. $}
Also, since \mbox{$ Y^{\pm 1}_{2\text{-part}} = \emptyset$}, it follows from~\eqref{eq:trap} and~\eqref{eq:triple}
that \mbox{$h(Y') = h(Y)$}, as desired.

Set \mbox{$V':= \xi(V) = V_0 \cup V_1{\cdot}y_\dag^{-1}$}.  It suffices to show that $V'$ is an \mbox{$(X,Y')$}-translator
with \mbox{$\vert V' \vert < \vert V \vert$}.

Since \mbox{$ Y^{\pm 1}_{2\text{-part}} = \emptyset$},  \eqref{eq:yunexc} says that
\mbox{$\xi (  v \xrightarrow{ {\cdot}y} {\mkern-9mu-}  v{\cdot} y )$}   equals   \mbox{$\xi( v)
 \xrightarrow{ {\cdot}y'}{\mkern-9mu-}\xi( v{\cdot} y)$} if
   \mbox{$y \in Y^{\pm1}-\{y_\ast\}^{\pm 1}$} and \mbox{$v \in \hat\iota_Yy,$}
and that\vspace{-4mm}
$$\xi( v   \xrightarrow{ {\cdot}y_\ast} {\mkern-9mu-}  v{\cdot} y_\ast)\,\,\,=\,\,\,
\begin{cases}
 \xi(v)  \overset{{\cdot}1:Y'}{\text{\textbf{--}}\cdots\mkern-2mu\to}{\mkern-10mu-} \xi(  v{\cdot}  y_\ast)
&\text{if }v \in  V_{ \hat\chi(\hat\iota_Y y_\ast)} \cap \hat\iota_Yy_\ast, \\
 \xi(v) \xrightarrow{{\cdot}y_\ast'}{\mkern-9mu-} \xi(v{\cdot}  y_\ast  )
&\text{if }v \in   V_{1- \hat\chi(\hat\iota_Yy_\ast)}  \cap \hat\iota_Y y_\ast.
\end{cases}  $$

By \eqref{eq:xunexc},
\mbox{$\xi ( v \xrightarrow{x{\cdot}} {\mkern-9mu-}   x{\cdot}v )$}  equals  \mbox{$ \xi( v)
 \xrightarrow{x{\cdot}}{\mkern-9mu-}\xi( x{\cdot}v)$}   if   \mbox{$x \in X^{\pm1},
 v \in \hat\iota_Xx,$} and
  \mbox{$  v \overset {x {\cdot}}{\rightleftharpoons }     x {\cdot} v $}    is not equal to
  \mbox{$v_\ast \overset {x_\ast{\cdot}}{\rightleftharpoons }   x_\ast{\cdot} v_\ast .$}
 It remains to  specify  \mbox{$\xi(v_\ast \xrightarrow{x_\ast{\cdot}}{\mkern-9mu-} x_\ast{\cdot}v_\ast)$}.
Clearly,  \mbox{$   v_\ast{\cdot}y_\ast  \overset {x_\ast{\cdot}}{\rightleftharpoons }
   x_\ast{\cdot}v_\ast{\cdot}y_\ast$}   is not equal to
\mbox{$v_\ast \overset {x_\ast{\cdot}}{\rightleftharpoons }      x_\ast{\cdot} v_\ast$};
 hence, \mbox{$ \chi(x_\ast{\cdot}v_\ast{\cdot}y_\ast)= \chi(v_\ast{\cdot}y_\ast) =
 1 - \hat\chi(\hat\iota_Yy_\ast)\in \{0,1\}$}.
We then have two subcases.

If \mbox{$\chi(v_\ast{\cdot}y_\ast)   = \chi(x_\ast{\cdot}v_\ast{\cdot}y_\ast)=1 - \hat\chi(\hat\iota_Yy_\ast) = 1$},
then \mbox{$y_\dag = y_\ast$}  and
\begin{align*}
&\xi(v_\ast{\cdot}y_\ast) = v_\ast{\cdot}y_\ast{\cdot}y_\dag^{-1} = v_\ast,
&&\xi(x_\ast{\cdot}v_\ast{\cdot}y_\ast) = x_\ast{\cdot}v_\ast{\cdot}y_\ast
{\cdot} y_\dag^{-1} =  x_\ast{\cdot}v_\ast,\\
&\xi(v_\ast) = v_\ast,
&&\xi(x_\ast{\cdot}v_\ast) = x_\ast{\cdot}v_\ast{\cdot}y_\dag^{-1} =x_\ast{\cdot}v_\ast{\cdot}y_\ast^{-1}.
\end{align*}
\vspace{-5mm}

\noindent
 Here, \mbox{$\xi:V \to  V'$}  is not injective, and we define \mbox{$\xi(v_\ast \xrightarrow{x_\ast{\cdot}} {\mkern-9mu-}x_\ast{\cdot}v_\ast)$}
  to be\vspace{-2mm}
$$\xi(v_\ast) \xrightarrow{x_\ast{\cdot}}{\mkern-9mu-} \xi(x_\ast{\cdot}v_\ast{\cdot}y_\ast)
\xrightarrow {{\cdot}y_\ast'^{-1}} {\mkern-9mu-}\xi(x_\ast {\cdot}v_\ast ) \vspace{-2mm}$$
in  \mbox{$\operatorname{Paths}(X \leftwreath V' \rightwreath Y')$}; notice that \mbox{$  \xi(x_\ast{\cdot}v_\ast{\cdot}y_\ast)
\xrightarrow {{\cdot}y_\ast'^{-1}}{\mkern-9mu-} \xi(x_\ast {\cdot}v_\ast )$}  equals \mbox{$\xi\bigl(   x_\ast{\cdot}v_\ast{\cdot}y_\ast
\xrightarrow {{\cdot}y_\ast^{-1}} {\mkern-9mu-} x_\ast {\cdot}v_\ast   \bigr)$}.

If \mbox{$\chi(v_\ast{\cdot}y_\ast) = \chi(x_\ast{\cdot}v_\ast{\cdot}y_\ast) =1 - \hat\chi(\hat\iota_Yy_\ast) = 0$},
then \mbox{$y_\dag = y_\ast^{-1}$}  and
\begin{align*}
&\xi(v_\ast{\cdot}y_\ast) = v_\ast{\cdot}y_\ast,
&&\xi(x_\ast{\cdot}v_\ast{\cdot}y_\ast) = x_\ast{\cdot}v_\ast{\cdot}y_\ast,
\\
&\xi(v_\ast) = v_\ast,
&&\xi(x_\ast{\cdot}v_\ast) = x_\ast{\cdot}v_\ast{\cdot}y_\dag^{-1} =x_\ast{\cdot}v_\ast{\cdot}y_\ast.
\end{align*}
\vspace{-5mm}

\noindent
Here, \mbox{$\xi:V \to  V'$}  is not injective,  and we define
\mbox{$ \xi(v_\ast \xrightarrow{x_\ast{\cdot}} {\mkern-9mu-}x_\ast{\cdot}v_\ast)$} to be\vspace{-2mm}$$ \xi(v_\ast) \xrightarrow{{\cdot}y_\ast' }{\mkern-9mu-} \xi( v_\ast{\cdot}y_\ast)
\xrightarrow {x_\ast{\cdot}}{\mkern-9mu-}\xi(x_\ast {\cdot}v_\ast )\vspace{-2mm}$$
in \mbox{$\operatorname{Paths}(X \leftwreath V' \rightwreath Y')$};
notice that  \mbox{$  \xi(v_\ast) \xrightarrow{{\cdot}y_\ast' }{\mkern-9mu-} \xi( v_\ast{\cdot}y_\ast)$}
equals
\mbox{$    \xi\bigl(   v_\ast
\xrightarrow {{\cdot}y_\ast} {\mkern-9mu-}   v_\ast{\cdot}y_\ast   \bigr)$}.

Since \mbox{$\xi:V \to  V'$} is surjective, but not injective, \mbox{$\vert V' \vert < \vert V \vert$}.
Notice also that \mbox{$y_\ast \in \langle V' \rangle$}; hence, \mbox{$V'$} generates $F$.

Let us examine the graphs \mbox{$ X \leftwreath V' \rightwreath Y' $},
\mbox{$ X \leftwreath V'  $}, and \mbox{$V' \rightwreath Y'$}.
From the form that $\xi$ takes here, we see that
 \mbox{$X \leftwreath V' \rightwreath Y' $} is obtained from
\mbox{$X \leftwreath V \rightwreath Y$} by  first
removing one edge, secondly reattaching it elsewhere,
and thirdly collapsing various edges.
The graph \mbox{$X \leftwreath V'$} is thus obtained from the tree
 \mbox{$X  \leftwreath V $} by first removing an edge,
leaving  components with vertex-sets $V_0$ and $V_1$,
secondly reattaching the edge elsewhere,
and thirdly    identifying  one or
 more vertices of $V_0$ with vertices of  $V_1$.
Hence, \mbox{$X \leftwreath V'$} is  connected, and therefore a tree.
The graph \mbox{$ V' \rightwreath Y'$}
is obtained from the tree \mbox{$ V \rightwreath Y $} by collapsing edges; hence,
\mbox{$ V' \rightwreath Y'$} is   a tree.  Thus, $V'$ is an $(X,Y')$-translator, and
  \mbox{$\operatorname{d}(X,Y) > \operatorname{d}(X,Y')$}.
\end{proof}

\bigskip

\centerline{\sc References}

\leftskip 18pt \parindent -18pt 

\bigskip

Warren\,\,Dicks:\,\,Groups, trees and projective modules.
Lecture Notes Math.\,\,\textbf{790}. Springer, Berlin (1980).

Warren\,\,Dicks:\,\,On free-group algorithms that sandwich a subgroup between free-product factors.
 J.\,\,Group Theory \textbf{17}, 13--28   (2014).

Warren\,\, Dicks:\,\,On Whitehead's first free-group algorithm, cutvertices, and free-product factorizations.
\url{https://arxiv.org/abs/1704.05338},  9 pages   (2017).

  M.\,\,Dehn:\,\,\"Uber die Topologie des dreidimensionalen Raumes. Math.\,\,Ann.\,\,\textbf{69}, 137--168
 (1910).

 Ralph H.\,\,Fox:\,\,Free differential calculus I.\,\,\,Derivation in the free group ring.
  Ann.\,\,of Math. \textbf{57}\,\,(2), 547--560  (1953).




S.\,\,M.\,\,Gersten:\,\,Fixed points of automorphisms of free groups.
Adv.\,\,in\,\,Math.\,\,\textbf{64}, 51--85    (1987).


Michael\,\,Heusener and Richard\,\,Weidmann:\,\,A remark on
Whitehead's lemma.  Preprint, 4 pages (2014).

P.\,\,J.\,\,Higgins and R.\,\,C.\,\,Lyndon:\,\,Equivalence of elements under automorphisms of a
free   group.   Queen Mary College Mimeographed Notes, London, 5 pages (1962).

P.\,\,J.\,\,Higgins and R.\,\,C.\,\,Lyndon:\,\,Equivalence of elements under automorphisms of a
free   group.
J.\,\,London Math.\,\,Soc.\,\,\textbf{8}, 254--258  (1974).


A.\,\,H.\,\,M.\,\,Hoare:\,\,On automorphisms of free groups\,\,I.\,\, J.\,\,London Math.\,\,Soc.\,(2)\,\,\textbf{38},
277--285  (1988).

A.\,\,H.\,\,M.\,\,Hoare:\,\,On automorphisms of free groups\,\,II.\,\,
J.\,\,London Math.\,Soc.\,(2)\,\,\textbf{42},  226--236 (1990).

Sava\,\,Krsti\'c:\,\,On graphs representing automorphisms of free groups.
Proc.\,\,Amer.\,\,Math.\,\,Soc. \textbf{107},   573--575  (1989).







J.\,\,Nielsen:\,\,\"Uber die Isomorphismen unendlicher Gruppen
ohne Relation.\,\,\,Math.\,\,Ann.\,\,\textbf{79}, 269--272 (1919).



Elvira\,\,Strasser\,\,Rapaport:\,\,On free groups and their automorphisms.\,\,Acta Math.\,\,\textbf{99},
139--163 (1958).






John\,\,R.\,\,Stallings:\,\,Automorphisms of free groups in graph theory and topology.  Talk given at the
  AMS meeting, SUNY, Albany, August 8--11  (1983).


  Richard\,\,Stong:\,\,Diskbusting elements of the free group. Math.\,\,Res.\,\,Lett.\,\,\textbf{4}, 201--210 (1997).

 J.\,\,H.\,\,C.\,\,Whitehead:\,\,On certain sets of elements in a free group.
 Proc.\,\,London  Math.\,\,Soc.  (2) \textbf{41}, 48--56  (1936a).

 J.\,\,H.\,\,C.\,\,Whitehead:\,\,On equivalent sets of elements in a free group. Ann.\,\,of Math.\,\,(2)
\textbf{37},  782--800  (1936b).

\leftskip 0pt \parindent 0pt

\end{document}